\documentclass[12pt]{amsart}
\usepackage {crimson}

  \usepackage{amsfonts,graphics,amsmath,amsthm,amscd, amssymb,amsmath,latexsym,bbm,bm,mathtools,mathrsfs}
  \usepackage{epsfig}
  \usepackage{wrapfig}
  \usepackage{flafter}
  \usepackage{verbatim,lipsum}
  \usepackage{enumerate}
  \usepackage{cases}

  \usepackage{url}
  \usepackage[pdfauthor={Franco Rota},
              pdftitle={Quot schemes in tilted hearts and moduli spaces of stable pairs},
              pdfsubject={Algebraic geometry, Quot scheme, stable pair},
              pdfkeywords={Quot pair tilt t-structure franco rota},
              colorlinks=true]{hyperref}

  \usepackage[english]{babel}
  \usepackage{graphicx,subfigure}
  \usepackage[a4paper,top=30mm,bottom=30mm,left=30mm,right=30mm]{geometry}
  \usepackage{pgf,tikz,pgfplots}
  \pgfplotsset{compat=1.14}
  \usetikzlibrary{cd,matrix,arrows,decorations.pathmorphing}
  \tikzcdset{row sep/normal=1em}
  
\usepackage{multicol}

\theoremstyle{plain}
\newtheorem{thm}{Theorem}[section]
\newtheorem{corollary}[thm]{Corollary}
\newtheorem{lemma}[thm]{Lemma}

\newtheorem{proposition}[thm]{Proposition}

\newtheorem{setup}[thm]{Setup}

\newtheorem*{thm*}{Theorem}
\newtheorem*{corollary*}{Corollary}
\newtheorem*{lemma*}{Lemma}
\newtheorem*{ld*}{Lemma/Definition}
\newtheorem*{proposition*}{Proposition}
\newtheorem*{assumption*}{Assumption}

\theoremstyle{definition}
\newtheorem{definition}[thm]{Definition}
\newtheorem{example}[thm]{Example}

\newtheorem*{definition*}{Definition}
\newtheorem*{example*}{Example}
\newtheorem*{xca*}{Exercise}
\newtheorem*{claim*}{Claim}
\newtheorem*{fact*}{Fact}
\newtheorem*{notation*}{Notation}
\newtheorem*{construction*}{Construction}
\newtheorem*{ack*}{Acknowledgements}
\newtheorem*{question*}{Question}
\newtheorem*{problem*}{Problem}
\newtheorem*{conjecture*}{Conjecture}

\theoremstyle{remark}
\newtheorem{remark}[thm]{Remark}

\newcommand{\C}{\mathbb{C}}
\newcommand{\Z}{\mathbb{Z}}
\newcommand{\Q}{\mathbb{Q}}
\newcommand{\R}{\mathbb{R}}

\newcommand{\pr}[1]{\mathbb P^{#1}}
\newcommand{\aff}[1]{\mathbb A^{#1}}

\newcommand{\HHom}{\mathcal{H}om}

\newcommand{\EXT}[3]{\mathrm{Ext}^{#1}\left( #2,#3\right) }

\newcommand{\HOM}[3]{\mathrm{Hom}_{{#3}}\left( #1,#2\right) }

\DeclareMathOperator{\rk}{rk}

\DeclareMathOperator{\SL}{SL}

\DeclareMathOperator{\coker}{coker}
\DeclareMathOperator{\Hom}{Hom}

\DeclareMathOperator{\Ext}{Ext}

\DeclareMathOperator{\rad}{rad}

\DeclareMathOperator{\td}{td}

\DeclareMathOperator{\im}{Im}

\DeclareMathOperator{\dR}{\mathbf{R}}

\DeclareMathOperator{\Bl}{Bl}

\DeclareMathOperator{\NS}{NS\,}
\DeclareMathOperator{\Pic}{Pic\,}
\DeclareMathOperator{\Quot}{Quot}
\DeclareMathOperator{\Hilb}{Hilb}
\DeclareMathOperator{\Spec}{Spec}
\DeclareMathOperator{\Proj}{Proj}
\DeclareMathOperator{\Supp}{Supp}

\DeclareMathOperator{\textch}{ch}												%chern classes
 
\newcommand{\cc}[1]{\text{c}_{{#1}}}
\newcommand{\ch}[1]{\textch_{{#1}}}

\DeclareMathOperator{\Sch}{Sch}
\DeclareMathOperator{\Sets}{Sets}

\DeclareMathOperator{\Coh}{Coh}

\newcommand{\sA}{\mathcal{A}}

\newcommand{\sO}{\mathcal{O}}
\newcommand{\sK}{\mathcal{K}}
\newcommand{\sL}{\mathcal{L}}
\newcommand{\sQ}{\mathcal{Q}}

\newcommand{\sT}{\mathcal{T}}
\newcommand{\sM}{\mathcal{M}}
\newcommand{\sC}{\mathcal{C}}
\newcommand{\sD}{\mathcal{D}}

\newcommand{\sE}{\mathcal{E}}
\newcommand{\sF}{\mathcal{F}}

\newcommand{\PP}{\mathcal{P}}

\newcommand{\sAt}{\mathcal{A}^\dagger}

\newcommand{\st}{\,|\,}					%such that
\usepackage{leftidx}						%allows sub/superscripts on the left of the symbol

\newcommand{\AAA}[1]{\textcolor{red}{#1}}

\DeclarePairedDelimiter{\set}{\lbrace}{\rbrace}
\DeclarePairedDelimiter{\ceil}{\lceil}{\rceil}

\DeclarePairedDelimiter{\pair}{\langle}{\rangle}

\DeclarePairedDelimiter{\abs}{\lvert}{\rvert}

\begin{document}

\definecolor{yqyqyq}{rgb}{0.5019607843137255,0.5019607843137255,0.5019607843137255}
\definecolor{uququq}{rgb}{0.25098039215686274,0.25098039215686274,0.25098039215686274}
\definecolor{uuuuuu}{rgb}{0.26666666666666666,0.26666666666666666,0.26666666666666666}

\title{Some Quot schemes in tilted hearts\\ and moduli spaces of stable pairs
       }
%\author{Franco Rota}
%\address{Department of Mathematics, Rutgers University, Piscataway, NJ}
%\curraddr{Department of Mathematics and Statistics,
%Case Western Reserve University, Cleveland, Ohio 43403}
%\email{rota@math.rutgers.edu}

%\thanks{2010 {\it Mathematics Subject Classification.} 18E30; 14H45, 14J33
%}
\author[F. Rota]{Franco Rota}
\address{FR: Department of Mathematics \\ Rutgers University \\ Piscataway,~
NJ 08854,~USA} 
\email{rota@math.rutgers.edu}

\subjclass[2010]{18E30; 14F05, 14D20}
\keywords{Quot scheme, families of t-structures, tilt, stable pairs, moduli spaces}
\thanks{The author was partially supported by NSF-FRG grant DMS 1664296.}
%18E30; derived,triangulated categories
% 14H45 special curves and curves of low genus
% 14J33 mirror symmetry
%14F05  	Sheaves, derived categories of sheaves and related constructions
%14D20  	Algebraic moduli problems, moduli of vector bundles

%    General info
%\subjclass[2000]{Primary 54C40, 14E20; Secondary 46E25, 20C20}

%\date{January 1, 2001 and, in revised form, June 22, 2001.}


%\keywords{Differential geometry, algebraic geometry}

\begin{abstract}
For a smooth projective variety $X$, we study analogs of Quot schemes using hearts of non-standard $t$-structures of $D^b(\Coh(X))$. The technical framework uses families of $t$-structures as studied in \cite{BLMNPS19}. We provide several examples and suggest possible directions of further investigation, as we reinterpret moduli spaces of stable pairs, in the sense of Thaddeus \cite{Th94} and Huybrechts--Lehn \cite{HL95}, as instances of Quot schemes.
\end{abstract}

\maketitle

%\setcounter{tocdepth}{1}
%\tableofcontents

%\mainmatter

\section{Introduction}

%Grothendieck introduced Hilbert schemes as a way of compactifying families of algebraic cycles in a given projective variety \cite{GroFGA}. Quot schemes answer the analogous problem of parametrizing flat families of quotients of a fixed sheaf over a projective variety $X$ (see \cite{GroSBIV}). The natural setting for the moduli problem associated with Quot schemes is the abelian category of coherent sheaves $\Coh(X)$: the objects parametrized are flat families of sheaves. 

Grothendieck's classical Quot scheme parametrizes flat families of quotients of a fixed sheaf over a projective variety $X$ (see \cite{GroSBIV}). The natural setting for this moduli problem is the abelian category of coherent sheaves $\Coh(X)$, which can be regarded as the heart of the standard $t$-structure of $D^b(\Coh(X))$.

In this note, we study examples of certain Quot spaces which generalize the classical notion: Quot spaces parameterize quotient objects of fixed numerical class $v$ in abelian subcategories $\sA\subset D^b(\Coh(X))$, which are the heart of different $t$-structures of $D^b(\Coh(X))$. We focus mostly on the case when $\sA$ is obtained from $\Coh(X)$ by tilting (once or twice), as these hearts play an important role in the study of Bridgeland stability conditions \cite{Bri07_triang_cat}. 
%giving rise to a \textit{collection of $t$-structures} in the sense of \cite[Sec. 11]{BLMNPS19}. 

The construction of Quot spaces hinges on the technical tools developed in \cite{AP06} and \cite{Pol07}, Quot spaces appear in \cite{Tod08} and \cite{BLMNPS19}, where the authors discuss representability and verify a valuative criterion for properness. We recall their results in Section \ref{sec_quot_spaces_collections} .
Section \ref{sec_first_examples} describes some examples of Quot spaces on curves and surfaces, and illustrates properness of Quot spaces by computing explicit limits of families of sheaves. \\

Quot spaces provide a very general working framework: Sections \ref{sec_Bradlow_pairs_on_curves} and \ref{sec_HL_pairs_as_quotients} reinterpret moduli spaces of stable pairs on curves and surfaces as examples of Quot spaces, they represent the bulk of this work. 
Section \ref{sec_Bradlow_pairs_on_curves} focuses on moduli spaces of \textit{Bradlow pairs} on a curve $C$, first introduced in \cite{Bra91} (here, we will mostly follow the setup of Thaddeus \cite{Th94} and Bertram \cite{Be97}). These are pairs $(E,s)$, where $E$ is a rank 2 vector bundle on $C$ and $s\colon\sO_C\to E$ is a holomorphic section of $E$, satisfying a GIT-type stability condition with respect to a parameter $\sigma$. Their stability is related via the Kobayashi--Hitchin correspondence to solvability of the vortex equation \cite{Bra91,GP94}.

In \cite{GP94}, the author shows that Bradlow's stability condition is equivalent to that of certain $\SL(2,\C)$-equivariant vector bundles on $C\times \pr 1$, and obtains a general construction of moduli spaces of pairs and an alternative proof to the above mentioned Kobayashi--Hitchin correspondence. The same technique is applied to the study of other objects such as holomorphic triples and chains in \cite{BGP96}.
% certain $\mathrm{PU}(2)$-monopole equations (see \cite{Dow00} and references therein).  

Here, we interpret moduli spaces of Bradlow pairs as Quot schemes. Then, we construct determinant line bundles on these Quot schemes, and study the birational geometry they induce (Section \ref{sec_determinant_bundles}). The results can be summarized as follows:

\begin{thm}[see Theorem \ref{thm_summary_curves}]
\label{thm_summary_curves_intro}
Let $\mathcal{M}_\sigma^{s}(2,D)$ be the moduli spaces of $\sigma$-stable pairs of rank 2 and fixed determinant $\sO_C(D)$ of degree $d$. Then there exists a one-parameter family of  Quot spaces $\sQ^\beta(v)$, with $\beta=\beta(\sigma)$, such that
\[ \sQ^{\beta}(v)\simeq  \mathcal{M}_\sigma^{s}(2,D). \]
%Otherwise, $\sQ^{\beta}(v)\simeq  \mathcal{M}_{\sigma'}^{s}(2,D)$, where $\sigma'=\sigma +\epsilon$ for some $0<\epsilon \ll 1$.
Moreover, determinant line bundles in $\Pic(\sQ^\beta(v))$ induce a sequence of moduli spaces and birational maps
\begin{equation}
    \label{eq_thaddeus_diagram}
    \begin{tikzcd}
\sQ^{\frac d2}(v) \arrow[dashed, "\phi_{\frac d2}"]{r} & \cdots \arrow[dashed, "\phi_{d-2}"]{r} &  \sQ^{d-2}(v) \arrow[dashed, "\phi_{d-1}"]{r} & \sQ^{d-1}(v)
\end{tikzcd}
\end{equation}
which coincides with the wall-crossing diagram obtained by Thaddeus in \cite{Th94}.
\end{thm}

We expect that the determinant bundle construction may have a crucial role in proving the projectivity of Quot spaces in general, and studying their birational geometry. This represents a direction of future investigation but remains out of the scope of this work.
On the other hand, Quot spaces can provide insight into birational geometry questions: for example, Bridgeland realizes threefold flops as instances of Quot schemes in \cite{Bri02_flops}.

In the same spirit as above, Section \ref{sec_HL_pairs_as_quotients} relates Quot schemes and \textit{Huybrechts--Lehn pairs} for curves and surfaces constructed in \cite{HL95}. A HL-pair is a pair $(E,s^*)$ where $E$ is a sheaf on a smooth projective $X$ and $s^*\colon E\to E_0$ is a map to a fixed sheaf on $X$.
Section \ref{sec_duality} makes the duality relation between HL- and Bradlow pairs explicit, at least in the case of curves. 

Moving to higher dimensions, one considers a tilt of $\Coh(X)$ with respect to Gieseker slope rather than Mumford slope, and a polynomial stability parameter $\alpha$. Our first observation is a simple consequence of \cite{HL95}:

\begin{corollary}[see Cor. \ref{cor_general_HL_comparison}]
Let $X$ be a smooth projective variety, and let  $\mathcal{M}_{\alpha}^{s}(P,E_0)$ denote the space of $\alpha$-stable HL-pairs on $X$ with fixed Hilbert polynomial $P$. Then there exists a one-parameter family of Quot spaces, $\sQ^\beta(v)$, with $\beta=\beta(\alpha)$ and $v=v(P)$, such that for generic values of $\beta$  we have
\begin{equation}
    %\label{eq_HL_iso_quot}
 \mathcal{M}_{\alpha}^{s}(P,E_0) \simeq  \sQ^\beta(v). \end{equation}
\end{corollary}

In the case that $X$ is a surface and $E_0=\sO_X$, we obtain a more precise statement (Theorem \ref{thm_summary_HLsurfaces}), which involves a family of quot schemes parametrized by a polynomial parameter $\delta$ and is analogous to  Theorem \ref{thm_summary_curves_intro} above.
\begin{comment}
\begin{thm}[see Theorem \ref{thm_summary_HLsurfaces}]
\label{thm_summary_HLsurfaces_intro}
Let $\mathcal{M}_\alpha^{s}(P,\sO_X)$ denote the moduli spaces of HL-$\alpha$-stable pairs of fixed reduced Hilbert polynomial $p$.
Then there exist polynomials $\delta$ and a family of Quot schemes $\sQ^\delta(v)$ such that 
\[ \sQ^\delta(v)\simeq \mathcal{M}_\alpha^{s}(p) \]
Moreover, we obtain a sequence of moduli spaces and birational maps
\begin{equation}
    \label{eq_HL_pairs_diagram}
    \begin{tikzcd}
\sQ^{\delta_e} \arrow[dashed, "\phi_{\delta_{e}}"]{r} & \cdots \arrow[dashed, "\phi_{\gamma_2}"]{r} &  \sQ^{d-2} \arrow[dashed, "\phi_{\gamma_{min}}"]{r} & \sQ^{\delta_{v}}.
\end{tikzcd}
\end{equation}
\end{thm}
\AAA{fix diagrams, names of spaces and maps!!}
\end{comment}
Like in the case of curves, we define determinant line bundles on the schemes $\sQ^\delta(v)$  and establish some positivity properties (Lemma \ref{lem_GRR_computation_surfaces}). \\

In the last section we consider a smooth projective surface $S$ and Quot spaces in a double tilt of $\Coh(S)$, and we give other applications and examples. Section \ref{sec_Th_pairs_on_surfaces} describes a class of quotients which are a higher-dimensional analog to the picture obtained in \cite{Be97} for Bradlow pairs: the Quot schemes coincide with certain linear series on $S$, and the first wall-crossing morphism is the blowup of a copy of $S$ in the linear series. 
%This is similar to the picture obtained for curves in \cite{Be97}, where wall-crossing morphisms are identified with repeated blowups of secant varieties.
A similar phenomenon is studied for Veronese embeddings of $\pr 2$ in \cite{Mar17}.

Pandharipande-Thomas stable pairs \cite{PT_curve_counting} are also an example of Quot spaces in a single tilt \cite[Lemma 2.3]{Bri11_Hall_curve_counting}. 
Section \ref{sec_PT_invariants} draws a comparison between certain Quot schemes in a double tilt and the moduli spaces of PT-pairs and illustrates a simple example. 

\subsection*{Conventions} We work over the field of complex numbers. 
If $Y$ is a scheme, we denote by $D_{qc}(Y)$ the unbounded derived category of $\sO_Y$-modules with quasi-coherent cohomologies.
For a Noetherian scheme $S$, we let $K^0(S)$ (resp. $K(S)$) denote the Grothendieck group of vector bundles (resp. coherent sheaves) on $S$.
If $X$ is a smooth quasi-projective variety, we denote  by $D(X)=D^b(\Coh(X))$ the bounded derived category of coherent sheaves on $X$, and by $N(X)\coloneqq K(X)/\rad \chi$ its numerical Grothendieck group. 

Given a vector bundle $E$ over a smooth projective variety $B$, write $\mathbb{P}(E)=\Proj(\text{Sym}^*E)$ for the projective bundle of one dimensional quotients of $E$. We write $\PP(E):=\mathbb{P}(E^{\vee})$ for the projective bundle of lines of $E$.

\subsection*{Acknowledgements} This note originates from a chapter of my Ph.D. dissertation. My gratitude goes to my doctoral advisor, Aaron Bertram, for introducing me to this subject and encouraging me to study it. I wish to thank Huachen Chen, Marin Petkovic, and Matteo Altavilla for our fruitful discussions. I'm grateful to Emanuele Macr\`i for teaching me parts of \cite{BLMNPS19}, and to Mariano Echeverria and Tom Bridgeland for pointing me to interesting references.

\section{Preliminaries}
\label{sec_quot_spaces_collections}

This section recollects preliminary facts about $t$-structures on triangulated categories, families of $t$-structures and Quotient spaces. We follow the setup in \cite{BLMNPS19}, we point the reader there for the results in full generality.

\subsection{Torsion pairs and $t$-structures on triangulated categories} First, we recall some terminology:

\begin{definition}
Given an abelian category $\mathcal A$, a \emph{torsion pair} (or \emph{torsion theory}) for $\mathcal A$ is a pair of full subcategories $(\sT,\sF)$ such that:
\begin{itemize}
\item $\HOM{\sT}{\sF}{}=0$;
\item for any $E\in \mathcal A$, there exists a short exact sequence
\begin{equation*}
0\to E' \to E \to E'' \to 0
\end{equation*}
where $E'\in \sT$ and $E''\in \sF$. 
\end{itemize}
%One often refers to $\sT$ (resp. $\sF$) as the \emph{torsion} part (resp. \emph{torsion-free} part) of the torsion pair. 
\end{definition}

\begin{definition}
A \emph{t-structure} on a triangulated category $\mathcal D$ is a pair of full triangulated subcategories $(\mathcal D^{\leq 0},\mathcal D^{\geq 0})$ such that:
\begin{itemize}
\item $\Hom(\sD^{\leq 0}, \sD^{\geq 1})=0$ (where we write $\mathcal D^{\leq i}:=\mathcal D^{\leq 0}[-i]$ and $\mathcal D^{\geq i}:=\mathcal D^{\geq 0}[-i]$);
\item For all $E\in \mathcal D$, there exists a triangle  
\begin{equation*}
E' \to E \to E'' \to E'[1];
\end{equation*}
where $E'\in \mathcal D^{\leq 0}$ and $E''\in \mathcal D^{\geq 1}$.
\item $\mathcal D^{\leq 0}[1] \subset \mathcal D^{\leq 0}$.
\end{itemize}
The \emph{heart} of a $t$-structure is the intersection $\mathcal{C}=\mathcal D^{\leq 0} \cap \mathcal D^{\geq 0}$. It is an abelian category and short exact sequences in $\sC$ are precisely triangles with objects in $\mathcal C$.
\end{definition}

%We say that a t-structure $(\mathcal D^{\leq 0},\mathcal D^{\geq 0})$ is \emph{non-degenerate} if $ \cap_n \mathcal D^{\leq n}= \cap_n \mathcal D^{\geq n}=\{ 0\}$, and that it is \emph{bounded} if $\cup_n \mathcal D^{\leq n}=\cup_n \mathcal D^{\geq n}= \mathcal D$. For more facts about t-structures, we refer the reader to \cite[Ch. 4]{Mil03}

%%
Given a torsion pair $(\sT,\sF)$ on an abelian category $\mathcal A$, we can define a $t$-structure on $D(\mathcal A)$ in the following way: set
\begin{equation*}
\begin{split}
\sD^{\leq 0}:= \left\lbrace E\in D(\mathcal A)\,|\,H^i(E)=0 \mbox{ for } i>0 \mbox{ and } H^{0}(E)\in \sT \right\rbrace \\
\sD^{\geq 0}:= \left\lbrace E\in D(\mathcal A)\,|\,H^i(E)=0 \mbox{ for } i<-1 \mbox{ and } H^{-1}(E)\in \sF \right\rbrace 
\end{split}
\end{equation*}
The heart of the above $t$-structure is an abelian category $\sAt=\left\langle \sF[1],\sT \right\rangle$, known as the \emph{tilt} of the abelian category $\mathcal A$ with respect to the torsion pair $(\sT,\sF)$. Explicitly, objects $E\in\sA^\dagger$ are characterized by the following conditions:
\begin{equation*}
    \begin{split}
        H^i(E)=0 \mbox{ for }i\neq 0,-1;\\
        H^0(E)\in \sT \mbox{ and } H^{-1}(E)\in\sF.
    \end{split}
\end{equation*}

\subsection{Collections of $t$-structures on a family}

In this section, we recall the main aspects of \cite[Section 11]{BLMNPS19}, which represents the theoretical framework for this work. Let $X$ be a smooth projective variety, and let $S$ a smooth quasi-projective variety. We will only consider constant families of varieties over $S$, so for every scheme $T$ over $S$, we write $X_T$ as a shorthand for the product $X\times T$.

\begin{definition}
A \textit{fiberwise collection of $t$-structures on }$D(X\times S)$ \textit{over} $S$ is a collection $\underline{\tau}=\set{\tau_s}_{s\in S}$ of $t$-structures $\tau_s$ on $D(X_s)$ for every (closed or non-closed) point $s\in S$. 
\end{definition}

This allows us to define flat objects. To do so, we will need to consider base change of families of $t$-structures, as in \cite[\S 5]{BLMNPS19}. For a scheme $Y$, we denote by $D_{qc}(Y)$ the unbounded derived category of $\sO_Y$-modules with quasi-coherent cohomologies. Nevertheless, this category does not play an important role in the rest of the paper. 

\begin{definition}\label{def_TauHat_t}
Let $\underline{\tau}$ be a fiberwise collection of $t$-structures on $D(X\times S)$ over $S$. Let $s$ be the image of a point $t\in T$ under a morphism $T\to S$: base change of $\tau_s$ gives rise to a $t$-structure $\hat{\tau}_t$ on $D_{qc}(X_t)$, write $(\sA_{qc})_t$ for its heart (see \cite[Theor. 5.3]{BLMNPS19}). We say that an object $E\in D(X_T)$ is \textit{$T$-flat with respect to $\underline{\tau}$} if $E_t\in (\sA_{qc})_t$ for all $t\in T$. 
\end{definition}

This allows us to introduce the notion of openness of flatness:

\begin{definition}
The fiberwise collection of $t$-structures $\underline{\tau}$  \textit{ satisfies
openness of flatness} if for every $S$-perfect object $E \in D(X\times S)$, the set
\[ \set{s\in S \st E_s\in (\sA_{qc})_s} \]
is open. Similarly, \textit{$\underline{\tau}$ universally satisfies openness of flatness} if for every $T \to S$ and every
$T$-perfect object $E \in D(X_T)$, the set
\[ \set{t\in T \st E_t\in (\sA_{qc})_t} \]
is open.
\end{definition}

\subsection{Quot spaces}
Suppose we have a fiberwise collection $\underline{\tau}$ of $t$-structures on $D^b(X\times S)$ over $S$ which universally satisfies openness of flatness. We can then define the Quot spaces:

%\AAA{what is $v$? Really necessary? equivalence relation?}

\begin{definition}\label{def_quot_space}
Let $V \in D(X\times S)$ be an $S$-perfect and $S$-flat object, and fix a class $v\in N(X)$. We denote by
\[ \Quot_{X/S}^{\tau}(V,v)\colon (\Sch/S)^{op} \longrightarrow \;\Sets \]
sending a $S$-scheme $T$ to the set of morphisms
$V_T\to Q$ such that:
\begin{enumerate}
    \item $Q\in D(X_T)$ is $T$-perfect and $T$-flat with respect to $\underline{\tau}$;
    \item The morphism $V_t\to Q_t \in (\sA_{qc})_t$ is surjective for all $t\in T$;
    \item The class $[Q_t]=v \in N(X_t)\simeq N(X)$ for all closed $t\in T$.
\end{enumerate}
    Given $T'\to T$, the corresponding map  $\Quot_{X/S}^{\tau}(V,v)(T)\to \Quot_{X/S}^{\tau}(V,v)(T')$ is given by the pull back of $V_T \to Q$ to $V_{T'}\to Q_{T'}$.
\end{definition}

\begin{proposition}[{\cite[Prop. 11.6]{BLMNPS19}}]
 $\Quot_{X/S}^{\tau}(V,v)(T)$ is an
algebraic space locally of finite presentation over $S$.
\end{proposition}

\begin{remark}
The numerical class $v$ does not appear in the definition of \cite[Section 11]{BLMNPS19}, but we may fix it here since we are working on a product $X\times S$. However, this does not guarantee that the Quot space is of finite presentation over $S$. In \cite{Tod08}, the author gives criteria for the boundedness of Quot spaces based on the existence of Bridgeland stability conditions.
\end{remark}

\begin{comment}
\AAA{do I need this?}
The correct setting to define Quot spaces is an "ambient" moduli space. This is the 2-functor $\sD_{pug}^b$: 
$$ \mathcal \sD_{pug}^b \colon  (\mathcal Sch/S)^{op} \longrightarrow \mathcal Set$$
sending a $S$-scheme $X$ to the groupoid $\sD_{pug}^b(X)$ consisting of objects $E\in D^b(X\times S)$ which are \textit{relatively perfect} and \textit{universally gluable}, i.e.
$$ \Ext^{-i}(E_s,E_s)=0 $$
for all $i>0$ and for all geometric points $s\in S$. It is an Artin stack locally of finite type over $\C$ \cite[Thm. 4.2.1]{Lie06}.

Then we can use $\underline{\tau}$ to define the stack of flat objects: \begin{definition}
 We denote by 
 \[\sM_\tau : (\mathcal Sch/S)^{op} \longrightarrow \mathcal Gpds\]
the functor whose value on $T \in (\mathcal Sch /S)$ consists of all objects $E \in \sD_{pug}^b(X_T /T)$ which are
$T$-flat with respect to $\tau$.
\end{definition}
\begin{lemma}[{\cite[Lemma 11.2]{BLMNPS19}}]
The functor $\sM_\tau$ is an algebraic stack locally of finite presentation over $S$, and
the canonical morphism $\sM_\tau \to \sD_{pug}^b$ is an open immersion.
\end{lemma}
\end{comment}

\subsection{Valuative criterion for Quot spaces}

The authors in \cite{BLMNPS19} investigate the following valuative criterion for properness for Quot spaces.

\begin{definition}\label{def_properness_of)quot}
Let $f\colon Z \to Y$ be a morphism of algebraic spaces. Let $R$ be a discrete valuation ring with fraction field $K$. Then we say that $f$ \textit{satisfies the strong existence part of the valuative criterion with respect to }$\Spec(R)\to Y$ if for any commutative square
\begin{equation}
\begin{tikzcd}
\Spec(K) \dar\rar &Z \dar{f}\\
\Spec(R) \rar \arrow[ur,dashed]{} &Y
\end{tikzcd}
\end{equation}
there exists a dashed arrow making the diagram commute. We say that $f$ \textit{satisfies the uniqueness part of the valuative criterion with respect to }$\Spec(R)\to Y$ if, there exists at most one dashed arrow making the diagram above commute. 
\end{definition}

To give the precise statement, we introduce the notion of a local $t$-structure on $D(X\times S)$.  

\begin{definition}
A $t$-structure
on $D(X\times S)$ is called \textit{$S$-local} if for every quasi-compact open $U \subset S$, there exists a $t$-structure
on $D(X\times U)$ such that the restriction functor $D(X\times S) \to D(X \times U)$ is $t$-exact.
\end{definition}

\begin{remark}
An $S$-local $t$-structure on $D(X\times S)$ (also called a \textit{sheaf of $t$-structures on $D(X\times S)$ over S} in \cite{AP06}) produces a fiberwise collection of $t$-structures via base change as in \cite[Theorem 5.3]{BLMNPS19}. The next definition is about the converse.
\end{remark}

\begin{definition}
Let $\underline{\tau}$ be a fiberwise collection of $t$-structures on $D(X\times S)$ over $S$, and let $T \to S$ be a morphism from a scheme $T$ which is quasi-compact with affine diagonal. We say $\underline{\tau}$ is
\textit{integrable over $T$} if there exists a $t$-structure $\tau_T$ on $D(X_T)$ such that for every $t \in T$ the $t$-structure
on $D(X_t)$ induced by base change agrees with the $t$-structure $\hat{\tau}_t$ on $D(X_t)$ from Definition \ref{def_TauHat_t}. In this situation, we say \textit{$\underline{\tau}$ integrates over
$T$ to the $t$-structure $\tau_T$}.
\end{definition}

Then we have the following:
\begin{proposition}\label{prop_properness}
Let $V\in D(X\times S)$ be an $S$-perfect and $S$-flat object. Let  $\Spec(R) \to S$ be a morphism, and assume that $\underline{\tau}$ integrates over $\Spec(R)$ to a
bounded $\Spec(R)$-local $t$-structure on $D(X_R)$ whose heart $\sA_R$ %
is noetherian.
%has a $Spec(R)$-torsion theory. \AAA{understand this part of the statement}
Then the morphism $ \Quot_{X/S}^{\tau}(V)\to S$ satisfies the strong existence and the uniqueness parts of
the valuative criterion with respect to $\Spec(R) \to S$.
\end{proposition}

\begin{proof}
This is a weaker restatement of Proposition 11.11 in \cite{BLMNPS19}. There, the assumption on the Noetherianity of $\sA_R$ is replaced by the weaker requirement that $\sA_R$ has a $\Spec(R)$-torsion theory (see \cite[Section 6]{BLMNPS19}).
\end{proof}

\begin{remark}
[Constant families of $t$-structures]
In what follows we will focus mainly on fiberwise collections of $t$-structures on $X\times S$ which are \textit{constant over $S$}, i.e. $\tau_s$ is the same $t$-structure on $D(X_s)\simeq D(X)$ for all closed $s\in S$. 

Examples of such collections were constructed in \cite{AP06}, starting from a given $t$-structure $\tau$ on $D(X)$. In the case that $\tau$ is \textit{noetherian} (i.e., it has a noetherian heart), then the corresponding constant collection $\underline{\tau}$ is $S$-local and satisfies openness of flatness. Moreover, the associated Quot spaces satisfy the above valuative criterion (a proof of Prop. \ref{prop_properness} in this setting appears in the author's PhD thesis \cite{Rota19thesis}). 
\end{remark}

\subsection{A premise to the remaining Sections}\label{Sec_premise}
What follows contains several examples of Quot spaces. Some of the $t$-structures considered do not clearly fit the framework of Section \ref{sec_quot_spaces_collections}. For example, openness of flatness is problematic for a constant family of $t$-structure constructed from a $t$-structure which is only \textit{tilted-Noetherian} (i.e. a $t$-structure that can be obtained by tilting a noetherian one. These are also called \textit{close to noetherian} in \cite{Pol07}) and not noetherian (see Remark \ref{rmk_noetherianity}).
However, it will be clear \textit{a posteriori} that the Quot spaces considered below are represented by projective schemes of finite type. %For this reason, it will be safe for us to ignore the representability issues and assume that the Quot functor $\underline{\Quot}_{X\times S}^{\sA}(V,v)$ is represented by a generalized Quot scheme $\Quot_{X\times S}^{\sA}(V,v)$.

\section{Some Quot spaces and wall-crossing phenomena}
\label{sec_first_examples}

Before moving on to a more systematic study of some Quot spaces, we present a few small examples for concreteness' sake. This also allows us to introduce some notation. 

First, we recall the definition of an important $t$-structure on $D^b(Y)$, in the case that $Y$ is a smooth projective variety.

\begin{definition}\label{def_Mumford_tilt}
If $Y$ is a smooth projective variety of dimension $n$ with an ample class $H$, we define the \textit{Mumford stability function} as
\[ Z(E)= - \deg (E) + i \rk(E) \]
and the associated slope to be $\mu(E)\coloneqq \frac{\deg(E)}{\rk(E)}$
where $\deg(E)\coloneqq \ch 1(E).H^{n-1}$ and $\rk(E)\coloneqq \ch 0(E)H^n$. The function $Z$ allows one to define a \textit{torsion pair}:
\begin{align*}
  \sT^{\beta} &= \{E\in \Coh(Y)\mid \mbox{ any quotient $E\twoheadrightarrow G\neq 0$ satisfies } \mu(G)>\beta\},\\
  \sF^{\beta} &= \{E\in \Coh(Y)\mid \mbox{ any subobject $0\neq F\hookrightarrow E$ satisfies }
\mu(F)\leq \beta\}
\end{align*}
and a corresponding $t$-structure $\tau_\mu^\beta$ whose heart $\Coh^\beta(Y)\coloneqq \pair{\sF_\beta[1],\sT_\beta}$ is called the tilt of $\Coh(Y)$ along the torsion pair $(\sT^\beta,\sF^\beta)$. 

For a base $S$, we denote by $\underline{\tau}_\mu^\beta$ the constant fiberwise collection of $t$-structures on $Y\times S$ obtained from $\tau_\mu^\beta$.
\end{definition}

\begin{remark}\label{rmk_noetherianity}
The heart $\Coh^\beta(Y)$ is noetherian if $\beta\in\Q$ \cite[Lemma 6.17]{MS17}, but it is not noetherian in general. This seems to be a known fact, but we give here an explicit counterexample (Example \ref{example_non_noeth}).
\end{remark}

\begin{example}[A non-noetherian tilt]\label{example_non_noeth}
The fact that $\Coh^\beta(C)$ is not noetherian in general is known but little is present in the literature about it, we provide here an explicit counterexample. Let $C$ be a smooth elliptic curve, and let $0<\beta\in\R\setminus \Q$. Let $Z$ be the Mumford stability function and consider its image $\Z^2\subset \C\simeq \R^2$. Denote by $\alpha_0=(0,1)$ the image of $\sO_C$ and by $w=(-1,0)$ that of a skyscraper sheaf.
The line $l\coloneqq\set{y=-\frac 1\beta x}$ has irrational slope in $\R^2$. 

\begin{claim*}\label{lem_irrational_line}
There exists a point $\alpha=(-d,r)\in\Z^2$ such that:
\begin{enumerate}[(i)]
    \item $0<\frac dr < \beta<\frac{d}{r-1}$;\label{part1}
    \item $\alpha$ is primitive;\label{part2}
    \item the triangle bounded by the line $l$, the line through $\alpha$, and the horizontal line $y=r$ contains no lattice point except from $\alpha$.\label{part3}
\end{enumerate}
\end{claim*}

\begin{proof}[Proof of claim]
The condition $\frac dr < \beta<\frac{d}{r-1}$ is equivalent to requiring that the vertical distance between $\alpha$ and $l$ is smaller than 1. This can be arranged: since $\beta$ is irrational, the translation orbit of $l$ in the flat torus $\R^2/\Z^2$ is dense, so there are lattice points arbitrarily close to $l$. 
Then, pick such an $\alpha'=(-d',r')$, and consider the triangle $T$ bounded by the line $l$, the line through $\alpha'$, and the horizontal line $y=r'$. If $T$ contains any lattice points $\gamma$, they still satisfy \eqref{part1}. One of these lattice points must satisfy \eqref{part3}, and hence also \eqref{part2}.
\end{proof}

Pick $\alpha_1\coloneqq\alpha=(-d_1,r_1)$ as in the Claim, and let $\mu_1=\frac{d_1}{r_1}$. We can apply a change of coordinates $\phi$ sending $\alpha_1 \mapsto \alpha_0$, and $w\mapsto w$, and apply the Claim once more with the transformed line $\phi(l)$, to find a lattice point $\alpha_2'$, write $\alpha_2=(-d_2,r_2)\coloneqq \phi^{-1}(\alpha_2')$ for its inverse image, it satisfies \eqref{part1}. As above, by possibly replacing $\alpha_2$ with another lattice point we may assume it satisfies \eqref{part2} and \eqref{part3}.
Iterating this process, we can construct a sequence of lattice points $\alpha_i=(-d_i,r_i)$, with slopes $\mu_i\coloneqq \frac{d_i}{r_i}$, satisfying:

\begin{enumerate}
    \item for all $i$,  $0<\mu_1<...<\mu_i < \beta<\frac{d_i}{r_i-1}$;\label{part1i}
    \item $\alpha_i$ are primitive; \label{part2i}
    \item the triangle $T_i$ bounded by $l$, the line through the origin and $\alpha_i$, and the horizontal line $y=r_i$ contains no lattice points except $\alpha_i$.\label{part3i}
\end{enumerate}

By Atiyah's classification of stable bundles on elliptic curves, for each $i$ there exists a slope-stable vector bundle $A_i$ on $C$, of rank $r_i$ and degree $d_i$ \cite{Ati57}. By stability, $A_i\in \sF^\beta$ for all $i$, so $A_i[1]\in \Coh^\beta(C)$.

The claim now is that there exists a sequence of maps 
\[ A_0\coloneqq\sO_C \xrightarrow{q_0} A_1 \xrightarrow{q_1} A_2 \xrightarrow{q_2} ... \xrightarrow{q_i} A_i \to ... \]
and that each $q_i[1]$ is an epimorphism in $\Coh^\beta(C)$. This will give the chain of non-trivial quotients which contradicts noetherianity: the chain of kernels $M_i\coloneqq \ker(q_i[1])\in \Coh^\beta(C)$ produces an infinite chain of subobjects
\[ M_0 \subsetneq M_1 \subsetneq M_2 \subsetneq ... \subsetneq \sO_C[1]. \]

First, the existence of such maps is a Riemann--Roch computation: 
\[ \hom(A_i,A_{i+1})\geq \chi(A_i,A_{i+1})=r_id_{i+1} - r_{i+1}d_i = (\mu_{i+1} - \mu_{i})r_ir_{i+1}>0, \]
so we may pick a non-zero $q_{i}\in \Hom(A_i,A_{i+1})$.
Each map $q_i$ has a kernel and a cokernel
\[ K_i \to A_i \xrightarrow{q_i} A_{i+1} \to H_i, \]
and $q_i[1] \colon A_i[1] \to A_{i+1}[1]$ is an epimorphism in $\Coh^\beta(C)$ if and only if the complex $K_i[1] \oplus H_i\in \Coh^\beta(C)$. Since $K_i$ is a subsheaf of $A_i$ we have $K_i\in\sF^\beta$, so we only need to check $H_i\in\sT^\beta$. This is true: by stability of $A_{i+1}$ we have that $\mu(Q)>\mu_{i+1}$ for all quotients $Q$ of $H_i$, and $Z(Q)\notin T_{i+1}$ by \eqref{part3i}, which shows $H_{i}\in \sT^\beta$ (see Fig. \ref{figure}).

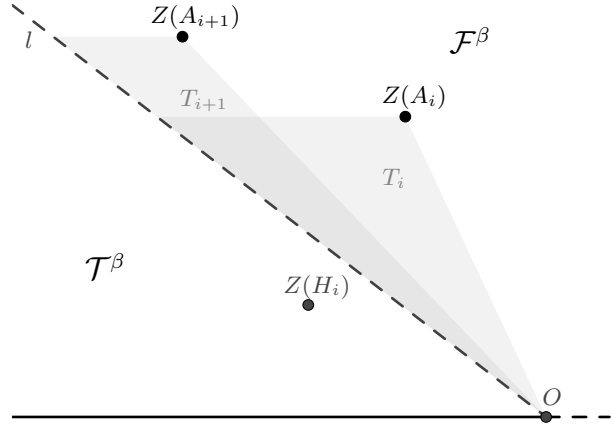
\begin{figure}[h]
\centering
\begin{tikzpicture}[line cap=round,line join=round,>=triangle 45,x=1cm,y=1cm]
\clip(-7.018965330510831,-0.5) rectangle (1,6);
\fill[line width=2pt,color=yqyqyq,fill=yqyqyq,fill opacity=0.1] (-5.119919996333709,3.982159997148441) -- (0,0) -- (-1.8518253157177245,3.9821599971484405) -- cycle;
\fill[line width=2pt,color=yqyqyq,fill=yqyqyq,fill opacity=0.1] (-6.483441191393253,5.042676482194753) -- (0,0) -- (-4.782199813872014,5.042676482194753) -- cycle;
\draw [line width=1pt,domain=-7.018965330510831:0] plot(\x,{(-0-0*\x)/-6});
\draw [line width=1pt,dash pattern=on 5pt off 5pt,domain=0:0.8335283159331679] plot(\x,{(-0-0*\x)/4});
\draw [line width=1pt,dash pattern=on 5pt off 5pt,color=uququq,domain=-7.018965330510831:0] plot(\x,{(-0--4.83*\x)/-6.21});
\begin{scriptsize}
\draw [fill=uuuuuu] (0,0) circle (2pt);
\draw[color=uuuuuu] (0.07868246473961923,0.27) node {$O$};
\draw[color=uququq] (-6.8,5) node {$l$};
\draw [fill=black] (-1.8518253157177245,3.9821599971484405) circle (2pt);
\draw[color=black] (-1.7350443721559907,4.242110037153214) node {$Z(A_i)$};
\draw [fill=black] (-4.782199813872014,5.042676482194753) circle (2pt);
\draw[color=black] (-4.612894179831394,5.300991022855274) node {$Z(A_{i+1})$};
\draw[color=yqyqyq] (-2,3.1622611111402232) node {$T_i$};
\draw[color=yqyqyq] (-4.5,4.200174156531351) node {$T_{i+1}$};
\draw [fill=uuuuuu] (-3.1278928810009834,1.4852594649578426) circle (2pt);
\draw[color=uuuuuu] (-3.0140887311228366,1.7259571998413896) node {$Z(H_i)$};
\end{scriptsize}
\draw[color=black] (-5.8,2) node {$\sT^\beta$};
\draw[color=black] (-1,5) node {$\sF^\beta$};
\end{tikzpicture}
\caption{The positions of $Z(A_i)$, $Z(A_{i+1})$ and $Z(H_i)$.}
\label{figure}
\end{figure}

\end{example}

Definition \ref{def_Mumford_tilt} suggests that varying the parameter $\beta$ may give rise to wall-crossing phenomena. This is illustrated in the next few sections.

\subsection{Atiyah flop as wall-crossing of Quot schemes}

Let $X\coloneqq \pr 1\times S$, where 
$$S\coloneqq\Hom_{\pr 1}(\sO(-2),\sO)=\aff 3$$
is the space of quadrics in $\pr 1$. Let $V\in D(X)$ be the universal quadric. For $\phi\in S$ a closed point, we have 
$ V_\phi \coloneqq [\sO(-2)\xrightarrow{\phi} \sO ]$. Let $\C_p$ denote a skyscraper sheaf supported at $p\in\pr 1$. For $\beta\in \R$, define by $\tau^\beta_s\coloneqq \tau_\mu^\beta$ on $X_s=\pr 1_s$, and let $\underline{\tau}^\beta$ be the corresponding fiberwise collection of $t$-structures on $X$ over $S$.
We describe the Quot space
\[\sQ^\beta_S \coloneqq\Quot_{X/S}^{\underline{\tau}^\beta}(V,[\C_p])(S),\]
as $\beta$ varies in an interval $I\coloneqq(-1-\epsilon , -1+\epsilon)$, denote by $\sQ^\beta_T$ its base change over $T \to S$.

If $\phi\neq 0$, then $V_\phi$ is quasi isomorphic to $\sO_Z$ where $Z$ is a length 2 subscheme of $\pr 1$, hence $\sQ^\beta_{\aff 3\setminus \set{0}}$ is a $2\colon 1$ cover of $\aff 3 \setminus \set{0}$ ramified over the quadric defined by the discriminant. This is independent of the value of $\beta$.

In other words, $\sQ^\beta_{\aff 3\setminus \set{0}}$ is isomorphic to the cone $X$ over a quadric in $\aff 4$ with the vertex removed. Such a cone admits two small resolutions $X^+$ and $X^-$ related by a flop. Let $-1<\beta^+\in I$ and $-1>\beta^-\in I$. The exact sequence in cohomology of a quotient in $\sQ_0^{\beta^+}$ reads
$$ 0\to \sO(-2) \to \sO(-1) \to \C_p \xrightarrow{0} \sO \xrightarrow{\sim} \sO \to 0 $$
%fits in a short exact sequence
%$$ 0\to \C_p \to \sO(-2)[1]\oplus \sO \to \sO(-1)[1] \to 0,$$
so that $\sQ^{\beta^+}_0\simeq \pr 1$. 

Similarly, the cohomology sequence of a quotient in $\sQ^{\beta^-}_0$ has the form
$$ 0\to \sO(-2) \xrightarrow{\sim} \sO(-2)\to 0 \to \sO(-1) \to \sO \to \C_p \to 0$$
and $\sQ^{\beta^-}_0\simeq \pr 1$. 
%$$ 0\to \sO(-2)[1]\oplus \sO(-1) \to \sO(-2)[1]\oplus \sO \to \C_p $$
This gives a description $X^-\simeq \sQ^{\beta^-}_S$ and $\sQ_S^{\beta^+}\simeq X^+$. 

\begin{remark}\label{rmk_s_times_beta}
In this example, we can interpreted the base space to be $S\times I$, where the collection of $t$-structures is given by $\tau_\mu^\beta$ on all closed points in $S\times \set{\beta}$. However, it helps the exposition to keep parameters in $S$ and $I$ separate.
\end{remark}

%\AAA{relate this to the more general wall crossing picture?}
%If $K\to V_0 \to Q$ is a quotient in $\sQ^{\beta^+}_0$, then its data is equivalent to a quotient $V_0\twoheadrightarrow \sO$ in $\Coh^{\beta^-}(\pr 1)$ with kernel $C$, and a map in $\PP(\Hom(C,\sO(-1)[1]))$. On the other hand, a quotient in $\sQ^{\beta^-}_0$ is the data of a surjection $V_0\twoheadrightarrow \sO$ with an extension class in $\PP(\Ext^1(\sO,\sO(-1)[1]))$. Therefore, $\sQ^{\beta^\pm}_0$ can be interpreted as $\pr 1$-bundles over the space $\Quot_{\pr 1,\Coh^{\beta^-}}(V_0,[\sO])\simeq \set{*}$ (cfr. the wall description in Section \ref{sec_quot_projectivity}).

\subsection{Point quotients of cubics on $\pr 1$}

This time, let $S$ be the space of cubics on $\pr 1$, $S\coloneqq\Hom_{\pr 1}(\sO(-3),\sO)=\aff 4$, and keep the rest of the notation as in the previous example. Let $V\in D(X)$ be the universal cubic and consider the relative Quot scheme 
$$ \sQ^\beta_S \coloneqq \Quot_{X/S}^{\underline{\tau}^\beta}(V,[\C_p])(S) $$
as $\beta$ varies in an interval $I\coloneqq(-2-\epsilon,-1+\epsilon)$.

Arguing as in the previous example, one sees that $\sQ^\beta_{\aff 3\setminus \set{0}}$ is a $3\colon 1$ cover of $\aff 4 \setminus \set{0}$ ramified over the quartic discriminant hypersurface. This is independent of the value of $\beta$.

We have that $V_0=\sO\oplus \sO(-3)[1]$. If $\beta > -1$, then $\sO(-1)\in\sF^\beta$. Consider a short exact sequence 
\[ (*)\colon \ \ 0\to K\to V_0 \to Q \to 0. \]
The only possibility for the associated exact sequence in cohomology is 
$$ 0\to \sO(-3) \to \sO(-1) \to \sO_Z \xrightarrow{0} \sO \xrightarrow{\sim} \sO \to 0 $$
where $Z$ is a length 2 subscheme of $\pr 1$ (if $H^0(Q)$ is torsion, then $H^0(K)$ surjects on some $\sO(d)$ with $d<0$, which is in not in $\sT^\beta$). 
%Again, this fits with the wall description in Section \ref{sec_quot_projectivity} with $\sQ_0^{-1+\epsilon}$ realized as a projective bundle $\PP(\Hom(\sO(-3)[1],\sO(-1)[1]))$ over $\Quot_{\pr 1,\Coh^{\beta^{-1-\epsilon}}}(V_0,[\sO])\simeq \set{*}$. 
After crossing the wall at $\beta=-1$, we see that the cohomology sequence associated to $(*)$ has the form
$$ 0\to \sO(-3) \to \sO(-3) \to 0 \to \sO(-1) \to \sO \to \C_p \to 0 $$

Then, $\sQ_0^{-1-\epsilon}$ is identified with $\pr 1\simeq \PP(\Hom(\sO(-1),\sO))$. The value $\beta=-2$ is only a potential wall: the object $\sO(-2)$ does not appear in the Harder-Narasimhan filtrations of any of the quotients appearing in $\sQ_0^{-2+\epsilon}$ and $\sQ^{-2-\epsilon}_0$: one checks that they have to have the same cohomology sequence on both sides of the wall).

\subsection{Non-emptiness across a wall}
We consider now the following example: let $S$ denote the space
$$S=\pr 5=\PP(\EXT{2}{\sO_{\pr 2}}{\sO_{\pr 2}(-5)}),$$
and let $\mathcal V$ be the corresponding universal extension on $X\coloneqq \pr 2\times S$. For $\beta\in \R$, define by $\tau^\beta_s\coloneqq \tau_\mu^\beta$ on $X_s=\pr 2_s$, and let $\underline{\tau}^\beta$ be the corresponding fiberwise collection of $t$-structures on $X$ over $S$. Let $v\coloneqq 5\left[\sO(-3)[1]\right]$, and consider the Quot spaces $$\Quot_{X/S}^{\underline{\tau}^\beta}(\mathcal V,v)$$ 
for $\beta\in(-3,-2]$. The general extension class in $S$ is represented by a complex 
\begin{equation*}
V\coloneqq \sO(-3)^5 \to \sO(-2)^5.
\end{equation*}
Fixing a line $l$ on $\pr 2$ we construct below a two-dimensional locus $S'$ in $S$ representing extension classes of the form 
\begin{equation*}
W\coloneqq \sO(-4)\oplus \sO(-3)^2 \to \sO(-2)^2\oplus \sO(-1).
\end{equation*}
Applying $\Hom(\sO,-)$ to the short exact sequence  
\[\sO (-5) \xrightarrow{l} \sO(-4)\to \sO_{l}(-4)\]
gives rise to an injection $\C^3 \simeq \Ext^1(\sO,\sO_{l}(-4)) \to \Ext^2(\sO,\sO(-5))$. Since we have $\Ext^1(\sO,\sO_{l}(-4))\simeq \Ext^1_l(\sO_l,\sO_l(-4))$, consider the open locus $S''\subset  \Ext^1_l(\sO_l,\sO_l(-4))$ parametrizing general extensions.  A class $\psi\in S''$ corresponds to a complex
\[ \psi\colon \qquad \sO_{l}(-4)\to \sO_l(-2)^2 \to \sO_l.  \]
The image of $\psi$ through $\Ext^1_l(\sO_l,\sO_{l}(-4))\to \Ext^2(\sO,\sO(-5))$ has the form $W$. Indeed, the total complex obtained by resolving every sheaf of $\psi$:
\begin{center}
\begin{tikzcd}
\sO(-5) \rar\dar & \sO(-3)^2 \rar\dar & \sO(-1)\dar \\
 \sO(-4) \rar\dar & \sO(-2)^2 \rar\dar & \sO\dar \\
 \sO_{l}(-4) \rar & \sO_l(-2)^2\rar & \sO_l,
\end{tikzcd}
\end{center}
has $W$ as its middle part. Finally, let $S'$ be the image of $S''$ in $\PP(\Ext^2(\sO,\sO(-5)))$.

For $\beta\in (-3,-2]$, we have $V,W\in \Coh^\beta(\pr 2)$. One sees immediately that $V$ admits an obvious family of quotients of class $v$. We now exhibit a quotient of class $v$ for $W$, proceeding by cases.

\textit{Case $\beta < -\frac52$.} The complex $W$ has a quotient $\left(\sO(-4)\oplus \sO(-3)^2\right)[1]$. In turn, the Euler sequence yields a morphism $\alpha$ whose cokernel
\[\left(\sO(-4)\oplus \sO(-3)^2\right) \xrightarrow{\alpha} \sO(-3)^5 \to T(-4) \to 0 \]
is a twist of the tangent bundle of $\pr 2$ ($\alpha$ is the identity on the $\sO(-3)$ summands). The map $\alpha$ is an epimorphism in $\Coh^\beta(\pr 2)$ if and only if $T(-4)\in \sT^\beta$, which is the case for $\beta < -\frac52$.

\textit{Case $\beta \geq -\frac52$.} Now we have $T(-4)\in \sF^\beta$, so the map $\alpha$ constructed above is not an epimorphism. Let 
\[Q\coloneqq\left(\sO(-4)\oplus \sO(-3)^2 \oplus T(-4)\right)[1],\]
we will produce an epimorphism $\alpha'\colon W\to Q$ in $\Coh^\beta(\pr 2)$. It suffices to construct an epimorphism $$\gamma\colon\sO(-2)^2\oplus \sO(-1) \to T(-4)[1].$$
indeed, one can combine the exact sequence of complexes
\begin{equation}\label{ses W}
 0\to \sO(-2)^2\oplus \sO(-1)\xrightarrow{\delta} W \xrightarrow{\alpha_1} \left(\sO(-4)\oplus \sO(-3)^2\right)[1] \to 0
\end{equation}
and the vanishing of 
$$\Ext^1(\left(\sO(-4)\oplus \sO(-3)^2\right)[1], Q )=0$$
to show that $\gamma$ factors through $\delta$ and produces a map  $\alpha_2\colon W\twoheadrightarrow T(-4)[1]$. Since $\gamma=\alpha_2\circ \delta$ is an epimorphism, so is $\alpha_2$. Then, $\alpha'\coloneqq \alpha_1 \oplus \alpha_2$ is the desired epimorphism. The only map $\sO(-1)\to T(-4)[1]$ corresponds to a non-trivial extension
\[ T(-4) \to E \to \sO(-1). \]
Since $T(-4)\simeq \Omega^1(4)\otimes \det(T(-4)) \simeq \Omega^1(-1)$, we recognize the Euler sequence and write $E=\sO(-2)^3\in \sT^\beta$.
In other words, there is a short exact sequence in $\Coh^\beta(\pr 2)$ $$0\to \sO(-2)^{\oplus 3}\to \sO(-1)\xrightarrow{\gamma'} T(-4)[1]\to 0$$
and $\gamma$ can be defined to be the composition of $\gamma'$ with the projection on the second factor.

\begin{remark}
This example can be viewed as an instance of the existence part of the valuative criterion for Quot spaces, if we interpreted the base space to be $S\times (-3,-2]$, as in Remark \ref{rmk_s_times_beta}. With this interpretation, what we described is the family of quotients parametrized by a curve through a point in $S'\times \set{-\frac{5}{2}}$.
\end{remark}

\subsection{Point quotients of plane conics and Gieseker-type walls}
Before the next example, we introduce another interesting $t$-structure.

\begin{definition}
\label{def_Gieseker_tilt}
Let $Y$ a smooth projective variety of dimension $e$, fix a positive rational polynomial $\delta$ of degree $e-1$ and an ample class $H$. Then, define the categories:
\begin{equation*}
    \begin{split}
\sT^\delta=\set{E\in\Coh(Y)\colon \mbox{ for all quotients } E\twoheadrightarrow G\neq 0, \mbox{ one has } p_G>\frac{t^e}{e!} + \delta }\\
\sF^\delta=\set{E\in\Coh(Y)\colon \mbox{ for all subsheaves } 0\neq F\subset E, \mbox{ one has } p_F\leq \frac{t^e}{e!} + \delta }
    \end{split}
\end{equation*}
where $p_E$ is the reduced Hilbert polynomial of $E$ (with respect to $H$), and we set $p_E=+\infty$ if $E$ is torsion (\cite[Def. 1.2.3]{HL10}). The pair $(\sT^\delta,\sF^\delta)$ defines a torsion pair on $\Coh(Y)$.
The resulting abelian category $\pair {\sF^\delta[1],\sT^\delta} $ is denoted $\Coh^\delta(X)$, and we denote by $\tau_G^\delta$ the corresponding $t$-structure on $D^b(Y)$.
\end{definition}

Consider the following example. Let $S\coloneqq\Hom_{\pr 2}(\sO_{\pr2}(-2),\sO_{\pr 2})=\aff 6$, denote $X\coloneqq \pr 2\times S$ and let $V\in D(X)$ be the universal object. Let $\C_p$ be a skyscraper sheaf at $p\in\pr 2$, and $I_p$ the corresponding ideal sheaf.
Denote by $\sQ_S^\beta$ the Quot space
\[\Quot_{X/S}^{\underline{\tau}^\beta}(V,[\C_p])(S)\]
as $\beta$ varies in an interval $I\coloneqq(-1-\epsilon , -1+\epsilon)$.

If $0\neq\phi\in S$, then $V_\phi\simeq \sO_C$ for a conic $C$, and $\sQ^{\beta}_{S\setminus \set{0}}$ can be compactified as the universal conic 
$$\sC \to \PP(H^0(\pr 2, \sO(2))).$$
Suppose now that $\phi=0$. For $\beta > -1$ (which implies $\sO(-1)\in \sF^\beta$), we see that the only possible cohomology sequence of a quotient $V_0\to Q$ in $\sQ^\beta_S$ is
$$(**) \ \ \sO(-2) \to \sO(-2) \to 0 \to I_p \to \sO \to \C_p. $$
Indeed, $H^0(Q)$ has to be torsion, which forces $H^{-1}(Q)$ to also be torsion, hence zero. Nothing changes as $\beta$ crosses the wall at $-1$; the possible quotients are exactly the same and yield $\sQ_0^{\beta}\simeq \pr 2$. 

Notice that, by definition, $\Coh^\beta(X)$ does not distinguish between torsion free sheaves with the same rank and degree (e.g. between $\sO$ and $I_p$). Let $\delta^\pm= p_{I_p}\pm \epsilon $ and consider $\Coh^{\delta^\pm}(\pr 2)$. A quotient in $\Coh^{\delta^-}(\pr 2)$ is readily seen to agree with the sequence $(**)$. 

On the other hand, observe that $\sO_{\pr 2}\in\Coh^{\delta^+}(\pr 2)$ but $I_p\in\Coh^{\delta^+}(\pr 2)[-1]$, so none of the $(**)$ are quotients in $\Coh^{\delta^+}(\pr 2)$. They get replaced by sequences 
$$ 0\to \sO(-2) \xrightarrow{\psi\neq 0} I_p \to I_p\otimes \sO_C \xrightarrow{0} \sO \to \sO $$
which are parameterized by a $\pr 4$-bundle over $\pr 2$ whose fiber over $p\in\pr 2$ is the linear series $\abs{I_p(2)}$ of conics through $p$. This realizes the scheme parameterizing quotients in $\Coh^{\delta^+}(\pr 2)$ as the blowup of $\sQ_S^{-1+\epsilon}$ along $\sQ_0^{-1+\epsilon}$. Remarkably, this further step of the wall-crossing is not detectable using Mumford slope only. This principle will be applied in Section \ref{sec_HL_pairs_as_quotients}.

\section{Quot spaces on curves: stable pairs}\label{sec_Bradlow_pairs_on_curves}

The goal of this Section is to compare the moduli spaces of stable pairs introduced and studied, for example, in \cite{Bra91,Th94,Be97} with the Quot functors defined above. Let $C$ be a smooth curve of genus $g\geq 2$. A \textit{Bradlow pair} (or just \textit{pair} if no confusion arises in the context) is the datum of a vector bundle with a global section. 

\subsection{Bradlow pairs as quotients} 
\label{sec_Bradlow_pairs_as_quotients}
 We start with the following lemma:

\begin{lemma}[{\cite[Ex. 2.1.12]{HL10}}]
Let $F$, $G$ be coherent sheaves on a smooth projective variety $X$, and let $E=\EXT 1FG$. Then there is a universal sequence on $\PP(E)\times X$ (with projections $p$ and $q$ to $\PP(E)$ and $X$, respectively), which reads
\begin{equation}\label{univext}
0\to q^*F\otimes p^*\sO_{\PP(E)}(1) \to \sE \to q^*G \to 0.
\end{equation}
\end{lemma}

Bradlow pairs admit stability conditions with respect to a parameter $\sigma\in \R$:

\begin{definition}\label{def_stable_bradlow_pair}
A Bradlow pair $(E,\phi\in H^0(E))$ is (semi)stable with respect to $\sigma$ if for all $F\subseteq E$:
\begin{enumerate}[(i)]
    \item $\mu(F) <(\leq) \mu(E) - \sigma\left(\frac{1}{\rk(F)}-\frac{1}{\rk(E)}\right)$ if $\phi\in H^0(F)$;\label{itm_bradlow_pair_factors}
    \item $\mu(F) <(\leq) \mu(E) + \sigma\left(\frac{1}{\rk(E)} \right)$  if $\phi\notin H^0(F)$.\label{itm_bradlow_pair_non_factors}
\end{enumerate}
\end{definition}

Fix a Chern class $(r,d)$ with $r,d\in\Z$ and a parameter $\beta\geq \frac{d}{r}$. Then we have:

\begin{lemma}\label{lem_Bradlow_pairs_are_quotients}
Suppose $E$ is a vector bundle on $C$ with $\ch{}(E)=(r,d)$, and $\sO_C\xrightarrow{\phi} E$ is a stable Bradlow pair with respect to $\sigma\coloneqq \beta r-d$. Then $\phi[1]\colon\sO_C[1]\to E[1]$ is an epimorphism in $\Coh^\beta(C)$. Conversely, an epimorphism $\sO_C[1]\to E[1]$ determines a semistable Bradlow pair.
\end{lemma}

\begin{proof}
Let $Q$ be a quotient of $\coker(\phi)$, and let $F\subset E$ be the kernel of the composition $$E\to\coker(\phi)\to Q.$$
Then $\phi$ factors through $F$. Condition \eqref{itm_bradlow_pair_factors} of Definition \ref{def_stable_bradlow_pair}, after substituting $\deg(F)=\deg(E)-\deg(Q)$ and $\rk(F)=\rk(E)-\rk(Q)$ everywhere, rewrites as $\mu(Q)>\mu(E)+ \frac{\sigma}{r}=\beta$. So $\coker(\phi)\in\sT^\beta$.

If $F$ is a subsheaf of $E$, then one among conditions \eqref{itm_bradlow_pair_factors} or \eqref{itm_bradlow_pair_non_factors} holds. Because we chose $\beta\geq \mu(E)$ we have $\sigma\geq 0$, so either condition implies $\mu(F)<\beta$, hence $E\in\sF^\beta$, and $\phi[1]$ is an epimorphism in $\Coh^\beta(C)$.
Reversing this argument we only get non-strict inequalities from the condition $E\in\sF^\beta$, so we obtain a semistable Bradlow pair. 
\end{proof}

We now follow the approach of \cite{Be97} and describe rank 2 Bradlow pairs directly in the framework of Quot schemes in a tilt. Let $v=(2,d)$ be the Chern character of a vector bundle of rank 2 and degree $d\geq 0$, and denote by $\sQ^\beta(v) \subset \Quot_{C/\text{pt}}^{\underline {\tau}^\beta}(\sO_C[1],-v)(\text{pt})$ the space of quotients whose determinant bundle is a fixed $\sO_C(D)$ (see Def. \ref{def_quot_space}). The results of this section can be summarized in the following theorem:

\begin{thm}\label{thm_summary_curves}
Let $\sigma=2\beta -d$, and denote by $\mathcal{M}_\sigma^{(s)s}(2,D)$ the moduli spaces of $\sigma$-(semi)stable Bradlow pairs of rank 2 and determinant $\sO_C(D)$. If $\mathcal{M}_\sigma^{ss}(2,D)\simeq \mathcal{M}_\sigma^{s}(2,D)$, then there is an isomoprhism
\[ \sQ^{\beta}(v)\simeq  \mathcal{M}_\sigma^{s}(2,D). \]
Otherwise, $\sQ^{\beta}(v)\simeq  \mathcal{M}_{\sigma'}^{s}(2,D)$, where $\sigma'=\sigma +\epsilon$ for some $0<\epsilon \ll 1$.
Moreover, there is a sequence of moduli spaces and birational maps as in diagram \eqref{eq_thaddeus_diagram} induced by interpolating determinant line bundles in $\Pic(\sQ^\beta)$. 
\end{thm}

\begin{proof}
The first part of this result is a combination of Lemma \ref{lem_Bradlow_pairs_are_quotients} with Propositions \ref{prop_boundary_cases} and \ref{prop_descr_of_walls}.
The second part of the statement follows from the computations in Sec. \ref{sec_determinant_bundles}.
\end{proof}

\begin{proposition}[Boundary cases]\label{prop_boundary_cases}
\begin{enumerate}[(i)]
\item The scheme $\sQ^\beta(v)$ is non-empty only if the parameter $\beta$ ranges in the interval $[\mu(E), \mu(\sO_C(D)))=\left[\frac d2,d\right)$. 
\item If $\beta\in \left[\frac d2, \ceil{\frac{d+1}{2}} \right)$, for all quotients $\sO_C\to E$ in $\sQ^\beta(v)$, $E$ is semistable. \label{itm_map_Qd2_to_semistable}
\item If $\beta\in \left[d-1, d \right)$, then $\sQ^\beta(v)=\PP(\Ext^1(\sO_C(D),\sO_C))$. \label{itm_map_Qd-1_to_determinants}
\end{enumerate}
\end{proposition}

\begin{proof}
If $\beta\geq d$, then the cokernel of $\sO_C\to E$ does not belong to $\sT^\beta$. Similarly, for $\beta < \frac d2$ we have that $E\not\in \sF^\beta$. 

If $\beta\in \left[\frac d2, \ceil{\frac{d+1}{2}} \right)$, all subobjects $F\subset E$ have slope smaller than or equal to $\frac d2$, because $\sF^\beta$ is closed under taking subobjects. Therefore $E$ is semistable.

Let $\beta\in \left[d-1, d \right)$ and consider an object $\sO_C\to E \to K$ in $\sQ^\beta(v)$. Then the cokernel $K$ is isomorphic to $\sO_C(D)$, and such sequences are parametrized by $\PP(\EXT{1}{\sO_C(D)}{\sO_C})$.
\end{proof}

\begin{proposition}[Description of walls]\label{prop_descr_of_walls}
Denote by $\Gamma$ the set $\Gamma=\left[\frac d2,d\right)\cap \Z$.  
\begin{enumerate}[(i)]
\item We have an isomorphism $\sQ^{\beta_1}(v)\simeq \sQ^{\beta_2}(v)$ if and only if $\beta_1$ and $\beta_2$ lie in a subinterval $\left[ c,c+1 \right)$, $c\in \Gamma$, or in $\left[\frac d2, \ceil{\frac{d+1}{2}} \right)$.
\item Suppose $c=d-n\in \Gamma$ and $\beta^- < c \leq \beta^+$ are real numbers in two neighboring intervals among the ones above. Denote by $W_c^\pm$ the subscheme of $\sQ^{\beta^\pm}(v)$ parametrizing $c$-critical objects. Then there exists an effective divisor $A\in \Pic(C)$ of degree $n$, and a diagram
\begin{center}
\begin{tikzcd}
 \sQ^{\beta^-}(v) \arrow[dashed, "\phi_c"]{r} & \sQ^{\beta^+}(v)  \\
  W^-_n \arrow[draw=none]{u}[sloped,auto=false]{\subset} \arrow[dr, "f^-_n"] & W^+_n \arrow[draw=none]{u}[sloped,auto=false]{\subset}\arrow[d, "f^+_n"] \\
 & \sQ^{\beta^-}(-(1,n))
\end{tikzcd}
\end{center}
where $\phi_c$ is an isomorphism away from the critical loci, $f^-$ and $f^+$ are projective bundles of rank $g-2+d-2n$ and $n-1$ respectively, with fiber $F^-=\PP(\EXT{1}{\sO_C(D-A)}{\sO_C(A)})$ and $F^+=\PP(\EXT{1}{\sO_{A}}{\sO_C(D-A)})$.
\end{enumerate}
\end{proposition}

\begin{proof}
A quotient $\sO_C\to E \to K$ is critical in $\sQ^{\beta}(v)$ only if $K$ has a quotient $T$ of slope $\beta$. This can only happen if $T$ has rank 1, so $\beta$ is an integer.

Let $\beta=d-n$. If an object is $\beta$-critical in $\sQ^{\beta^-}(v)$ then there exists an effective divisor $A$ of degree $n$ such that the following diagram commutes
\begin{center}
\begin{tikzcd}
  & \sO_C(A) \dar \rar & \sO_A \dar \\
  \sO_C \arrow[ru, "f"]\rar & E \dar \rar & K \arrow[dl]\\
& \sO_C(D-A)  &
\end{tikzcd}
\end{center}
The diagram is the datum of an object $\sO_C[1]\to \sO_C(A)[1]$ in $\sQ^{\beta_-}(-(1,n))$, and an extension in $F^-=\PP(\EXT{1}{\sO_C(D-A)}{\sO_C(A)})$. Similarly, a $\beta$-critical object in $\sQ^{\beta^+}(v)$ corresponds to a diagram 
\begin{center}
\begin{tikzcd}
  & \sO_C(D-A) \dar \arrow[dr] & \\
  \sO_C \arrow[rd, "f"]\rar & E \dar \rar & K \arrow[d]\\
& \sO_C(A) \rar & \sO_{A}
\end{tikzcd}
\end{center}
which is the datum of a map $\sO_C\to \sO_C(A)$ and of an extension class in $F^+$.%=\PP(\EXT{1}{\sO_{A}}{\sO_C(D-A)})$.
\end{proof}

Thus, we obtain diagram \eqref{eq_thaddeus_diagram} and observe that every map in it is an isomorphism off of the $W_n^\pm$.

\begin{remark}
\begin{enumerate}[(i)]
    \item In the setting of Prop. \ref{prop_descr_of_walls}, we have $(d-n)^- >n$. Then it is straightforward to check that $\sQ^{\beta^-}(-(1,n))$ can be identified with the space of all effective line bundles of degree $n$. This space is the symmetric product $C^{(n)}$.% In other words, $\Quot^{\beta_-}(\sO[1], -P(\sO(A)))\simeq C^{(n)}$.
    \item Combining this and Proposition \ref{prop_descr_of_walls}, we see that $W_1^-$ is a divisor and $f_1^-$ is a blow-down, while the maps $\phi_c$ are isomorphisms off a codimension two locus for $c\leq d-2$. Therefore we can identify the Picard groups of the $\sQ^\beta(v)$, for $\beta\in [\dfrac d2,d-1)$. Their generators are (the strict transforms of) $H$ and $E$, which are the hyperplane class in $\sQ^{d-1}(v)=\PP(\EXT{1}{\sO_C(D)}{\sO_C})$ and the exceptional divisor of $f_1^-$, respectively. 
    \item In \cite{Th94}, the author obtains a diagram like \eqref{eq_thaddeus_diagram} for moduli spaces of Bradlow pairs and traces sections of $\sO(H)$ across the diagram to give a proof of the Verlinde formula.
\end{enumerate}
\end{remark}

\subsection{Determinant line bundles}\label{sec_determinant_bundles}

In this section we construct determinant line bundles on $\sQ^\beta(v)$ which govern the transformations in diagram \eqref{eq_thaddeus_diagram}. Determinant line bundles can be constructed on moduli spaces of sheaves via Fourier-Mukai transforms, we recall here the construction and refer the reader to \cite{HL10} for the details.  

Let $\sK$ be a flat family of sheaves on a smooth projective variety $X$ parametrized by a Noetherian scheme $S$. Denote by $[\sK]$ its class in $K^0(X\times S)$. Since the projection $p\colon X\times S\to S$ is smooth, the assignment  
\[ p_!\left[F\right]=\sum (-1)^{i}\left[R^i p_*F \right]\]
is a homomorphism $p_!\colon K^0(X\times S)\to K(S)$ (see \cite[Cor. 2.1.11]{HL10}).

\begin{definition}[{\cite[Definition  8.1.1]{HL10}}]\label{def_det_line_bundles}
Let $\lambda_\sK\colon K(X)\to \Pic (S)$ be the composition 
$$ K(X)\xrightarrow{q^*} K^0(X\times S) \xrightarrow{\cdot [\sK]} K^0(X\times S) \xrightarrow{p_!} K(S) \xrightarrow{\det} \Pic(S),$$
where $\det$ is the determinant map and $q\colon X\times S\to X$ is the other projection. 
\end{definition}

This is the Fourier-Mukai transform associated to $\sK$ composed with the determinant map.

To investigate these line bundles, we will use the Grothendieck-Riemann-Roch formula. Keeping the above notation, for $[F]\in K^0(X\times S)$,  we have 
\begin{equation}\label{eq_GRR}
    \ch{}(p_!\left[F\right])=p_*(\ch{}\left[F\right]\cdot q^*\td(X)).
\end{equation}

Now we switch back to the case of stable pairs on curves. Fix any $\beta\in \left[\frac d2,d\right)$, and let $p,q$ be the projections on the two factors from $\sQ^{\beta}(v)\times C$. Denote by $[\sK]$ the class of the flat family of the cokernels of the sheaf maps $\sO_C\to E$ in $K(\sQ^{\beta}(v)\times C)$. Following Definition \ref{def_det_line_bundles}, given a class $w=r[\sO_C]+t[\C_p]\in K(C)$ we can define a line bundle $\lambda_\sK(w)$ on $\sQ^{\beta}(v)$. In particular, we are interested in the following two choices: 
\begin{equation*}
    \begin{split}
w_{1}\coloneqq [\sO_C]-d[\C_p]\cdot\td(C)^{-1}\\
w_{2}\coloneqq [\sO_C]-[\C_p]\cdot\td(C)^{-1}.
    \end{split}
\end{equation*}
and in the line bundles 
\begin{equation*}
\sM_{1}\coloneqq \lambda_{\sK}(-w_{1})  \qquad 
\sM_{2}\coloneqq \lambda_{\sK}(-w_{2})
\end{equation*}
(which do not depend on the choice of $p\in C$).

\begin{comment}
\begin{equation*}
    \begin{split}
\sL_{det}:= -\cc{1}(p_!([\sK]\otimes q^*(1,-d[\C_p]))) & \quad &
\sL_{quot}:= -\cc{1}(p_!([\sK]\otimes q^*(1,-[\C_p])))
    \end{split}
\end{equation*}
\end{comment}

For $d-n \in \Gamma$, the critical loci $W_n^\pm$ are projective bundles over $C^{(n)}$. We denote by $\sO_{W_n^\pm}(1)$ the corresponding relatively ample line bundles. 

\begin{lemma}\label{lem-restricting_line_bundles_to_fibers}
Suppose $d-n\in \Gamma$. Then, up to tensoring by a line bundle of $C^{(n)}$, the line bundles $\sM_1$  and $\sM_2$ restrict to $W_n^\pm$ according to the following table:
\begin{center}
\begin{tabular}{ l | c c }
   & $\sM_1$ & $\sM_2$ \\ \hline
 $W^-_n$ & $\sO_{W_n^-}(-n)$ & $\sO_{W_n^-}(d-n-1)$ \\
 $W^+_n$  & $\sO_{W_n^+}(n)$ & $\sO_{W_n^+}(-d+n+1)$ \\
\end{tabular}
\end{center}
\end{lemma}

%On the fibers $F^-_n$ and $F^+_n$ they restrict to the hyperplane class multiplied by the following coefficients:
%\begin{center}
%\begin{tabular}{ l | c c }
%   & $\sL_{det}$ & $\sL_{quot}$ \\ \hline
% $F^-_n$ & $-n$ & $(d-n)-1$ \\
% $F^+_n$  & $n$ & $1-(d-n)$ \\
%\end{tabular}
%\end{center}

\begin{proof}
This is an application of the Grothendieck-Riemann-Roch theorem (see \eqref{eq_GRR}). For example, we compute the restriction of $\sM_1$ on a fiber $F\coloneqq F^-_n$ of $W_n^-\to C^{(n)}$. Denote by $L$ the hyperplane class of $F$, by $P$ the class of a point on $C$, and by $l$ and $p$ the respective pull-backs to $F\times C$. The universal sequence (\ref{univext}) on $F\times C$ reads
\[ q^*\sO_C(A) \to \sE \to   q^*\sO_C(D-A)\otimes p^*\sO_{F}(-1)  \]
(we twist the sequence by $p^*\sO_{F}(-1)$ since this preserves the composition $\sO_C \to \sO_C(A) \to E$, which is fixed data in the fiber). Then we can write the first three entries of the Chern character
\begin{equation*}
\begin{split}
\ch{\leq 2}(\sK) &=\ch{\leq 2}(\sE)-\ch{\leq 2}(\sO_{F\times C})\\ 
 &=(1,np,0) + \left(1,-l, \frac{l^2}{2}\right)(1,(d-n)p,0) - (1,0,0) \\
 &=\left(1, -l + dp, \frac{l^2}{2} + (n-d)lp\right).
\end{split}
\end{equation*}

Pick $w=(1, -tp)\td(C)^{-1}$, then 
\[ [\ch{}(\sK) \otimes q^* w]_{\leq 2} = \left(1, -l + (d-t)p, \frac{l^2}{2} + (t+n-d)lp \right)\cdot q^*\td(C)^{-1}.\]

Pushing forward along $p$ (apply the projection formula, and use the fact that $C$ has dimension 1), we get $-c_1(p_!([\sK]\otimes q^*w))=(d-n-t)L$, therefore
\[ \sM_{1|F}=(-n)L. \]
Sequence (\ref{univext}) reads
\[ q^*\sO_C(D-A)\otimes p^*\sO_{F^+_n}(1)  \to \sK_+ \to  q^*\sO_A  \]
 on $F^+_n\times C$. A similar computation yields in this case 
\[-c_1(p_!([\sK]\otimes q^*w))=(t+n-d)L.  \]
\end{proof}

As a consequence of Lemma \ref{lem-restricting_line_bundles_to_fibers}, one immediately sees:

\begin{proposition}\label{prop_Bradlow_line_bundles}
Let $r=\frac{\beta-1}{d-1}$. Then the line bundle $\sL_r:=r\sM_1 + (1-r)\sM_2\in \Pic(\sQ^\beta(v))$ is ample on $\sQ^\beta(v)$ for generic $\beta$. For $\beta=c$ a critical value, multiples of $\sL_r$ induce the birational transformations $\phi_c$ of diagram \eqref{eq_thaddeus_diagram}.
%line bundle $\sL_r:=r\sM_1 + (1-r)\sM_2\in \Pic(\sQ^\beta(v))_{\mathbb R}$ is ample for $\beta\in \left[\frac d2,d\right] - \Gamma$, and if $\beta=d-n\in \Gamma$ then $\sL_r$ is trivial exactly along $W^\pm_n$.
\end{proposition}

\begin{remark}
In virtue of part \eqref{itm_map_Qd2_to_semistable} of Prop. \ref{prop_boundary_cases} there is an additional map 
$$ \phi\colon\sQ^{\frac d2}(v) \rightarrow M(2,D) $$
where $M(2,D)$ is the moduli space of semistable vector bundles with rank 2 and determinant $\sO_C(D)$ on $C$. Let $\sE$ be the family of semistable vector bundles corresponding to $\phi$. One checks (see \cite[Chp. 8]{HL10} and \cite[Sec. 4]{BM15}) that multiples of the line bundle $\sL\coloneqq\lambda_{\sE}\left( -(1,-\frac d2) \right)$ are ample on $M(2,D)$. For $\beta=\frac d2$ we have $\sL_r=\lambda_{\sK}\left( -(1,-\frac d2) \right)$. Since $[\sE]$ and $[\sK]$ only differ in rank, a Grothendieck-Riemann-Roch computation shows $\sL\equiv\sL_r$ on $\sQ^d(v)$, so that $\phi$ is induced by multiples of $\sL_r$.

We can perform the constructions of this section without fixing the determinant of $E$, but only its topological type. In that case, part \eqref{itm_map_Qd-1_to_determinants} of Prop. \ref{prop_boundary_cases} shows that there is a map $$\psi\colon \sQ^{d-1}(v) \to \Pic^d(C) $$
mapping a quotient $\sO_C[1]\to E[1]$ to the determinant of $E$. Observe that for $\beta=d$ we have $\sL_r=\sM_1$ is the determinant line bundle on $\Pic^d(C)$ constructed in \cite{HL10}. Then $\phi$ is induced by multiples of $\sL_r$. 

This illustrate the fact that $\sL_r$ coincides with the line bundles governing the Thaddeus' flips, described explicitly, for example, in \cite{Be97}.
\end{remark}

\subsection{Duality and the definition of pairs}\label{sec_duality}
It is common in the literature to find definitions for pairs which are "dual" to each other, e.g. \cite{Th94} considers maps \textit{from} a fixed sheaf, while \cite{HL95} studies maps \textit{to} a fixed sheaf. In this section, we reconcile the two notions in the case of a curve $C$, and show that they are related by an involution of $D(C)$ (i.e. an autoequivalence which is its own quasi-inverse).

Let $X$ be a smooth projective variety of dimension $e$. Define the functor $\mathsf D\colon D^b(X) \to D^b(X)$ as
\begin{equation}
    \label{eq_dual}
    \mathsf D(E)\coloneqq \dR\HHom(E,\sO_X)[e].
\end{equation}

\begin{lemma}
The functor $\mathsf D$ is an involution of $D(X)$.
\end{lemma}
\begin{proof}
This is because of the natural isomorphism $E\simeq \mathsf D(\mathsf D(E))$ (see \cite[Sec. 3.3]{Huy06}).
\end{proof}

First, we compare hearts of bounded t-structures:

\begin{proposition}\label{prop_duality_cohbeta_curves}
Let $C$ be a smooth curve, $\beta\in\R\setminus \Q$. Then $\mathsf D(\Coh^\beta(C))=\Coh^{-\beta}(C)$.
\end{proposition}

\begin{proof}
Suppose $L$ is a slope-semistable vector bundle on $C$, then $L^\vee = \mathsf D(L)[-1]$ is also semistable, of slope $-\mu(L)$. In fact, a maximal destabilizing quotient $L^\vee\twoheadrightarrow Q^\vee$ must be a vector bundle of slope $\mu(Q^\vee) < \mu(L^\vee)$. Dualizing the quotient we obtain $Q\subset L$ and an inequality $\mu(Q)>\mu(L)$, which contradicts semistability of $L$.
Sheaves in $\sF^\beta$ are obtained by repeated extensions of semistable vector bundles of slope $\mu \leq \beta$ (note that equality is never attained because slopes are rational), so for $F\in \sF^\beta$ one gets $\mathsf D(F[1])=F^\vee \in \sT^{-\beta}$. 
Elements of $\sT^\beta$ are sheaves $T\oplus L$ where $T$ is torsion and $L$ is obtained by repeated extensions of semistable vector bundles of slope $\mu > \beta$. Then, 
$$  \mathsf D(T\oplus L) = T\oplus  L^\vee[1]$$
where $L^\vee \in \sF^{-\beta}$. This shows $\mathsf D(\Coh^\beta(C))\subset \Coh^{-\beta}(C)$, so they must coincide \cite[Ex. 5.6]{MS17}.
\end{proof}

Then we investigate quotients:

\begin{proposition}\label{prop_duality of pairs_curves}
Let $C$ be a smooth curve and $0<\beta\in\R\setminus\Q$. If $\sO_C\to G$ is a quotient in $\Coh^{-\beta}(C)$ then the dual map $\mathsf D(G)\to \sO_C[1]$ is an inclusion in $\Coh^{\beta}(C)$. 
\end{proposition}

\begin{proof}
All subobjects of $\sO_C[1]$ in $\Coh^{\beta}(C)$ are of the form $L\oplus \sO_Z$ where $L$ is a vector bundle in $\sT^\beta$ and $Z$ is a finite length subscheme of $C$. Dually, all quotients of $\sO_C$ in $\Coh^{-\beta}(C)$ have the form $L'[1]\oplus \sO_{Z'}$ where $L'\in\sF^{-\beta}$ and $Z'$ is a finite length subscheme. These objects match under the equivalence $\mathsf D$, as shown in the proof of Prop. \ref{prop_duality_cohbeta_curves}.
\end{proof}

This effectively identifies the moduli spaces on a curve constructed using either notion of pair: a short exact sequence $\sO_C(D) \to \sO_C[1] \to E[1]$ in $\Coh^\beta(C)$ (corresponding to the Bradlow pair $\sO_C\to E$) is dual to an exact sequence $E(-D)\to \sO_C\to \sO(-D)[1]$ in $\Coh^{-\beta}(C)$ (corresponding to the dual pair $E(-D)\to \sO_C$).

\section{Huybrechts-Lehn stable pairs as quotients}\label{sec_HL_pairs_as_quotients}

In this section we interpret stable pairs as defined in \cite{HL95} as generalized quotients, and extend aspects of Section \ref{sec_Bradlow_pairs_on_curves} to higher dimensions, with particular attention to the case of surfaces.

\subsection{Huybrechts-Lehn pairs}

A \textit{Huybrechts-Lehn pair} (\textit{HL-pair} for brevity, or just pair if no confusion arises) is defined in \cite{HL95} to be the datum of a map $E\xrightarrow{a} E_0$ from a coherent sheaf $E$ to a fixed sheaf $E_0$ (and is in this sense "dual" to Thaddeus' definition). The stability condition is adjusted from that of Bradlow pairs to involve all coefficients of the reduced Hilbert polynomial, as explained below.
Let $X$ be a smooth, irreducible, projective variety of dimension $e$. Let $\sO(1)$ be a very ample line bundle and pick $H\in\abs{\sO(1)}$. For a coherent sheaf $E$ on $X$, denote $\deg(E)\coloneqq c_1(E).H^{e-1}$. Let $P_E(t)\coloneqq \chi(E(t))$ denote the Hilbert polynomial of $E$ with respect to $H$, we assume it has degree $e$. Let $\alpha$ be a polynomial of degree $e-1$ with rational coefficients such that $\alpha >0$, we write $\alpha=\sum_{k< e}\alpha_k t^{e-k}$ (inequalities between polynomials are meant asymptotically, i.e. they are realized for $t\ll 0$).

\begin{definition}\label{def_HL_stability_pairs}
A pair $(E,a)$ is said to be (semi)stable with respect to $\alpha$ if:
\begin{enumerate}[(i)]
    \item $\rk(E)P_G <(\leq) \rk(G)(P_E - \alpha)$ for all nonzero subsheaves $G\subseteq \ker(a)$;\label{it_HL_semist_kernel}
    \item $\rk(E)P_G <(\leq) \rk(G)(P_E - \alpha) + \rk(E)\alpha$ for all nonzero subsheaves $G \subseteq E$.\label{it_HL_semist_E}
\end{enumerate}
\end{definition}

Similarly, one defines slope-stability (also called $\mu$-stability):

\begin{definition}\label{def_mu_semistability_pairs}
A pair $(E,a)$ is said to be $\mu$-(semi)stable with respect to $\alpha_1\in \R$ if:
\begin{enumerate}[(i)]
    \item $\rk(E)\deg(G) <(\leq) \rk(G)(\deg(E) - \alpha_1)$ for all nonzero subsheaves $G\subseteq \ker(a)$;\label{it_mu_semist_kernel}
    \item $\rk(E)\deg(G) <(\leq) \rk(G)(\deg(E) - \alpha_1) + \rk(E)\alpha_1$ for all nonzero subsheaves $G \subseteq E$.\label{it_mu_semist_E}
\end{enumerate}
\end{definition}

There are implications:
\begin{center}
    $\mu$-stable $\implies$ stable $\implies$ semistable $\implies$ $\mu$-semistable.
\end{center}

\begin{lemma}[{\cite{HL95}}]\label{lem_ker(a)_tor-free}
Suppose $(E,a)$ is $\mu$-semistable. Then $\ker(a)$ is torsion-free.
\end{lemma}

\begin{proof}
Suppose $G\subseteq \ker(a)$ is torsion. Then, $\rk(G)=0$ and condition \eqref{it_mu_semist_kernel} of Definition \ref{def_mu_semistability_pairs} implies that $P_G=0$ (because $\rk(E)>0$), hence $G=0$.
\end{proof}

A family of pairs parametrized by a Noetherian scheme $T$ consists of a sheaf $\sE$ on $X\times T$ and a morphism $a\colon \sE\to p_X^*E_0$. $(\sE,a)$ is a family of (semi)stable pairs if $(\sE_t,a_t)$ is (semi)stable on $X\times \set{t}$ for all $t\in T$. Then, one can define functors
\begin{equation}%\label{eq_moduli_functor}
    \begin{split}
        \mathcal{M}_\alpha^{(s)s}(P,E_0): & (Sch/k)^{op} \rightarrow  Sets
    \end{split}
\end{equation}
associating to a scheme $S$ the set of isomorphism classes of families of (semi)stable pairs $a\colon\sE\to p_X^*E_0$ parametrized by $S$, with $P_{\sE_t}=P$ for all $t\in T$.

The main result of \cite{HL95} is:

\begin{thm}[{\cite[Thm. 1.21]{HL95}}]
Let $X$ be a smooth projective variety of dimension one or two. Then there is a projective $k$-scheme $M_\alpha^{ss}(P,E_0)$ which corepresents $\mathcal{M}_\alpha^{(s)s}(P,E_0)$. Moreover, there is an open subscheme $M_\alpha^{s}(P,E_0)\subset M_\alpha^{ss}(P,E_0)$ representing the subfunctor $\mathcal{M}_\alpha^{s}(P,E_0)$.
\end{thm}

%Let $X$ be a smooth projective variety of dimension $e$ with an ample class $H$, $E_0$ a coherent sheaf on $X$, and $\alpha$ a positive rational polynomial of degree $e-1$. Let $P$ be a polynomial.In what follows 

Now we move back to the general setting, and compare the functors $\mathcal{M}_{\alpha}^{(s)s}(P,E_0)$ with certain Quot functors in a tilt of $\Coh(X)$. Fix a positive rational polynomial $\delta$ of degree $e-1$, and recall the categories $\sT^\delta$, $\sF^\delta$ and $\Coh^\delta(X)$ of Definition \ref{def_Gieseker_tilt}. Then we have:

\begin{lemma}\label{lem_implication_pairs_injections}
Pick $E_0\in\sT^\delta$. Suppose $a\colon E\to E_0$ is a stable pair with respect to $\alpha\coloneqq p_E-\delta\rk(E)$. Then $a$ is a monomorphism in $\Coh^\delta(X)$. Conversely, a monomorphism $a\colon E \to E_0$ in $\Coh^\delta(X)$ is a semistable pair with respect to $\alpha$. 

Now pick $\beta\in \R$ and suppose $E_0\in\sT^\beta$, and $a\colon E\to E_0$ is a $\mu$-stable pair with respect to $\alpha_1\coloneqq \mu_E-\beta\rk(E)$. Then $a$ is a monomorphism in $\Coh^\beta(X)$. Conversely, a monomorphism $a\colon E \to E_0$ in $\Coh^\beta(X)$ is a $\mu$-semistable pair with respect to $\alpha_1$. 
\end{lemma}

\begin{proof} The proof is analogous to that of Lemma \ref{lem_Bradlow_pairs_are_quotients}.
The strict version of condition \eqref{it_HL_semist_E} in Definition \ref{def_HL_stability_pairs} holds if and only if $E\in\sT^\delta$. On the other hand, since $\ker(a)$ is torsion free by Lemma \ref{lem_ker(a)_tor-free}, condition \eqref{it_HL_semist_kernel} of Definition \ref{def_HL_stability_pairs} can be rewritten as
\begin{equation}
    \label{eq_pair_to_tilt_ker(a)}
    p_G < p_E - \frac{\alpha}{\rk(E)}=\delta
\end{equation}
for all nonzero $G\subseteq \ker(a)$, which implies $\ker(a)\in \sF^\delta$. 
Then, the exact sequence
$$ \ker(a) \to E\xrightarrow{a} E_0 \to \coker(a) $$
exhibits $a$ as an injection in $\Coh^\delta(X)$.
Conversely, $\ker(a)\in \sF^\delta$ is equivalent to the non-strict version of \eqref{eq_pair_to_tilt_ker(a)}, which yields semistability for $(E,a)$. The statements about slope-stability are completely analogous.
\end{proof}

\begin{comment}%%% at the critical values, we still don't have that the two things coincide, because of the inequalities being strict on one case and not the other.
\begin{corollary}Fix a Hilbert polynomial $P$ and a sheaf $E_0$, let $\frac{P_0}{\deg P!}$ be the leading coefficient of $P$. There is an isomorphism of functors
$$ \mathcal{M}_{\alpha}^{s}(P,E_0) \simeq \Quot^\delta_X(P_{E_0}-P,E_0), $$
where $\delta=\frac{P-\alpha}{P_0}$.
\end{corollary}
\end{comment}

For generic values of the parameter, the notions of $\mu$-stability and $\mu$-semistability coincide:

\begin{lemma}[{\cite[Lemma 2.2]{HL95}}]\label{lem_HL_slope_critical_vals_pairs} There exists a discrete set of rational numbers $\set{\eta_i}$ such that for $\delta_1\in(\eta_i,\eta_{i+1})$ every $\mu$-semistable pair with respect to $\alpha$ is in fact $\mu$-stable and $\mu$-stability only depends on $i$.
\end{lemma}

In other words, writing $\delta_1=\mu(E)-\beta\rk(E)$, and letting $v\in N(X)$ such that $\int_S v\cdot \ch{}(\sO_X(tH))\td(S)=P_{E_0}-P$, we can compare moduli spaces of HL-pairs and Quot schemes away from the critical values of the parameters:

\begin{corollary}
\label{cor_general_HL_comparison}
There is a discrete set of rational numbers $\Gamma$ such that if $\beta\in\R\setminus \Gamma$ and $E_0\in\sT^\beta$ then we have an isomorphism of functors
\begin{equation}
    %\label{eq_HL_iso_quot}
 \mathcal{M}_{\alpha}^{s}(P,E_0) \simeq  \Quot_{X/\text{pt}}^{\underline{\tau}^\beta}(E_0,v). \end{equation}
\end{corollary}

\subsection{Huybrechts--Lehn pairs on surfaces}

In this section we study the behaviour of moduli spaces of rank 2 $\delta$-stable pairs $E\to \sO_S$ on a smooth projective surface $S$, when $\delta=\delta_1t+\delta_2$ varies. Fix an ample class on $S$.

For a topological type $w=(w_0,w_1,w_2)$ in $H^*(S)$ (with $w_0\neq 0$), we denote 
$$\delta_w(t)\coloneqq - \frac{t^2H^2}{2} - \int_S \frac{w}{w_0} \cdot \ch{}(\sO_S(t))\cdot \td(S). $$ 
The polynomial $\delta_w(t)$ is the linear truncation of the reduced Hilbert polynomial of a sheaf of Chern character $w$ on $S$, and can be thought of as representing the Gieseker slope of the class $w$. Fix Chern characters $e\coloneqq \ch{}(E)=(2,L,c)$ and $v\coloneqq e-(1,0,0)=(1,L,c)$ on $S$, where $L$ is a line bundle of degree $-d\leq 0$.
For $\delta$ as above, consider the scheme
\[\sQ^\delta(v)\coloneqq \Quot_{S/\text{pt}}^{\underline{\tau}^\delta}(\sO_S,-v)(\text{pt}).\] The discussion below has the same structure of Section \ref{sec_Bradlow_pairs_on_curves}, and its results are summarized in the following Theorem.

\begin{thm}\label{thm_summary_HLsurfaces}
Let $P_e$ denote the Hilbert polynomial of an object of class $e$, and $p_e$ the reduced one. Define $\alpha=-2\delta -p_e$.  If $\mathcal{M}_\alpha^{ss}(P_e,\sO_S)\simeq \mathcal{M}_\alpha^{s}(P_e,\sO_S)$, then there is an isomorphism
\[ \sQ^{\delta}(v)\simeq  \mathcal{M}_\alpha^{s}(P_e,\sO_S). \]
Otherwise, $\sQ^{\delta}(v)\simeq  \mathcal{M}_{\alpha'}^{s}(2,D)$, where $\alpha'=\alpha +\epsilon$ for some $0<\epsilon \ll 1$.
Moreover, we obtain a sequence of moduli spaces and birational maps
\begin{equation}
    \label{eq_HL_pairs_diagram}
    \begin{tikzcd}
\sQ^{\delta_e} \arrow[dashed, "\phi_{\delta_{e}}"]{r} & \cdots \arrow[dashed, "\phi_{\gamma_2}"]{r} &  \sQ^{\delta_{v+(0,0,1)}} \arrow[dashed, "\phi_{\gamma_{min}}"]{r} & \sQ^{\delta_{v}}.
\end{tikzcd}
\end{equation}
\end{thm}

The sequence \eqref{eq_HL_pairs_diagram} of moduli spaces and birational maps is analogous to Thaddeus' diagrams. Lemma \ref{lem_GRR_computation_surfaces} indicates candidates for determinant line bundles inducing the birational maps, in the same spirit as Lemma \ref{lem-restricting_line_bundles_to_fibers}.

\begin{proposition}[Boundary cases]\label{prop_boundary_cases_HL}
\begin{enumerate}[(i)]
\item The scheme $\sQ^\delta(v)$ is non-empty only if the polynomial parameter $\delta$ satisfies $\delta_{v}\leq \delta < \delta_{e}$. 
\item \label{itm_map_Qdelta_max_to_semistable} If $\delta\geq  \delta_{(1,L/2,(c-1)/2)}$, all quotients $ E\to \sO \to Q $ in $\sQ^\delta(v)$ have $E$ Gieseker semistable.
\item \label{itm_map_Qdelta_min_to_determinants} If $\delta_{v} \leq \delta < \delta_{v+(0,0,1)}$, the kernel $E$ of every quotient in $\Coh^{\delta}(S)$ fits in a short exact sequence of sheaves
$$ B\to E\to \sO_S  $$
where $B$ is a Gieseker semistable sheaf of class $v$. Let $M(v)$ denote the moduli space of Gieseker semistable sheaves of class $v$, there is a map $\sQ^{\delta}(v)\to M(v)$ whose fiber above $B$ is the space $\PP(\Ext^1(\sO_S,B))$. 
\end{enumerate}
\end{proposition}

\begin{proof}
If $\delta < \delta_{v}$, the sheaf kernel $H^{-1}Q$ of the map $E\xrightarrow{\alpha}\sO_S$ must have either degree less than $d$, or $\ch{2}(K)<c=\ch{2}(E)$. Either of these conditions implies that $\im\alpha$ cannot be a subsheaf of $\sO_S$. On the other hand, we have that if $\delta\geq \delta_{e}$ then no sheaf of class $e$ belongs to $\Coh^\delta$. 

If $\delta\geq  \delta_{(1,L/2,(c-1)/2)}$ all sheaf quotients of $E$ must have Gieseker slope bigger than $\delta_e$ since they lie in $\sT^\delta$, hence $E$ is semistable. 

If $\delta_{v} \leq \delta < \delta_{(1,L,c+1)}$, then $B\coloneqq H^{-1}Q$ must have slope $\delta_{v}$, and $E\to \sO_S$ is a surjection. Moreover, all subsheaves of $B$ are in $\sF^{\delta_{v}}$ and therefore do not destabilize $B$. 
\end{proof}

Now we construct the set of critical values for $\delta$. The idea is the following: every map $E\to \sO_S$ admits a sheaf kernel $H^{-1}Q$. Let $(1,\beta,n)\in N(S)$ be the numerical class of a subsheaf $B\subset H^{-1}Q$. Since $H^{-1}Q\in \sF^\delta$, $B$ is a rank 1 torsion free sheaf, hence it satisfies the Bogomolov-Gieseker inequality: $2n\leq \beta^2$. Similarly, the cokernel $G$ of $B \to E$ satisfies the Bogomolov-Gieseker inequality, and is a subsheaf of $\sO_S$. In other words, 
\[ 2(c-n)\leq (\ch 1(L)-\beta)^2. \]
Then the set $\Gamma$ below is the set of critical values for $\delta$: for a divisor class $\beta\in\NS(X)$ satisfying $\beta.H\in [-d,-\frac d2]$, let $I_\beta$ be set of integers in $\left[ c - \frac{(\ch 1(L)-\beta)^2}{2}, \frac{\beta^2}{2} \right)$. We denote by $\Gamma'$ the set $\left\lbrace (1,\beta,n) \right\rbrace$ where $\beta\in\NS(X)$ is as above and $n\in I_{\beta}$. Set 
$$\Gamma\coloneqq \set{ \delta_\gamma \st \gamma\in\Gamma' }.  $$
We can think of $\Gamma$ as a discrete set partitioning an ordered family of polynomials, varying from $\delta_{v}$ to $\delta_e$, in neighboring intervals.

\begin{proposition}[Description of walls]\label{prop_descr_of_walls_surfaces}
\begin{enumerate}[(i)]
\item We have an isomorphism $\sQ^{\delta_1}(v)\simeq \sQ^{\delta_2}(v)$ if and only if $\delta_1$ and $\delta_2$ are not separated by an element of $\Gamma$.
\item Suppose $\gamma\coloneqq (1,\beta,n)\in \Gamma'$ and $\delta^- < \delta_\gamma \leq \delta^+$ are polynomials in two neighboring intervals. Denote by $W_{\gamma}^\pm$ the subscheme of $\sQ^{\delta^\pm}(v)$ parametrizing $\gamma$-critical objects. Then there is a diagram
\begin{center}
\begin{tikzcd}
 \sQ^{\beta^-}(v) \arrow[dashed, "\phi_{\gamma}"]{r} & \sQ^{\beta^+}(v)  \\
  W^-_{\gamma} \arrow[draw=none]{u}[sloped,auto=false]{\subset} \arrow[dr, "f^-_{\gamma}"] & W^+_{\gamma} \arrow[draw=none]{u}[sloped,auto=false]{\subset}\arrow[d, "f^+_{\gamma}"] \\
 & M(\gamma) \times \Hilb(X,P)  
\end{tikzcd}
\end{center}
where $M(\gamma)$ denotes the moduli space of semistable sheaves on $X$ of class $\gamma$ (and the polynomial $P$ is constructed below), the map $\phi_{\gamma}$ is an isomorphism away from the critical loci, and the fibers of $f^-$ and $f^+$ above a point $(B,[\sO_S\to \sO_Z])$ are the projective spaces $F^-=\PP(\Hom(B,\sO_Z)$ and $F^+=\PP(\Ext^1(I_{Z},B)$.
\end{enumerate}
\end{proposition}

\begin{proof}
If a quotient $E\to \sO \to Q$ in $\sQ^{\delta}(v)$ is critical, there is an injection $B[1]\to Q$ where $B$ satisfies $\delta_{\ch{}(B)}=\delta$. In other words, there is a diagram
\begin{center}
\begin{tikzcd}[row sep=small]
 B \dar \ar{dr} & & &\\
H^{-1}Q \rar & E \rar & \sO_S \rar & H^0Q
\end{tikzcd}
\end{center}
By construction of $\Gamma'$, $\ch{}(B)$ takes values in $\Gamma'$, hence $\delta \in \Gamma$. This proves the first part.

The critical quotients in $\sQ^{\delta^+}(v)$ are those which fit in a diagram 
\begin{center}$(*)\quad$
\begin{tikzcd}[row sep=small]
& & B[1] \dar\\
E \dar \rar & \sO_S \ar[equal]{d} \rar & Q \dar \\
C \rar \dar  &\sO_S \rar & A \\
B[1] & &
\end{tikzcd}
\end{center}
where all rows and arrows are exact sequences in $\Coh^{\delta^+}(X)$, and $B$ is a Gieseker semistable sheaf with $\delta_{\ch{}(B)}=\delta_\gamma$. 

One sees that the only possibility for this to happen is for $B$ to coincide with $H^{-1}Q$: the sequence in cohomology of the rightmost column of $(*)$ reads
$$ B\to H^{-1}Q \to H^{-1}A \to 0 \to H^0Q \to H^0A.  $$
Since $\rk B=\rk H^{-1}Q=1$,  $H^{-1}A$ must be zero and $B\simeq H^{-1}Q$. This also implies that $A=H^0A\simeq H^0Q$ is the structure sheaf of a subscheme $Z$ of $X$ with Hilbert polynomial $P=P_Q - P_{B[1]}$. Then, $C$ is the ideal sheaf corresponding to $Z$. The diagram $(*)$ can then be reconstructed from $B$ and $Z$ with the datum of an extension $B\to E\to I_Z$. 

Similarly, critical quotients of $\sQ^{\delta^-}(v)$ fit in a diagram 
\begin{center}$(**)\quad$
\begin{tikzcd}[row sep=small]
& & B \dar\\
I_Z \dar \rar & \sO_S \ar[equal]{d} \rar & \sO_Z \dar \\
E' \rar \dar  &\sO_S \rar & Q' \\
B & &
\end{tikzcd}
\end{center}
where rows and columns are short exact sequences in $\Coh^{\delta^-}(X)$. In this case, it's enough to specify a map $B\to \sO_Z$ to reconstruct the diagram from $B$ and $Z$. 

In this way, we get the diagram \eqref{eq_HL_pairs_diagram} and see that every birational map is an isomorphism off of the walls $W_\gamma^\pm$.
\end{proof}

Now we sketch a Grothendieck--Riemann--Roch computation on certain determinant line bundles, in analogy with Section \ref{sec_determinant_bundles}. Fix $\delta_{v}\leq\delta<\delta_e$, and let $\left[ \sE \right]$ be the class of the universal kernel in $K^0(\sQ^\delta(v)\times X )$, and let $p,q$ denote the projections on the two factors. For $\alpha\in K(X)$, consider the line bundle
\[ \lambda_\sE(\alpha) = \det \left(p_!\left(\left[\sE\right]\cdot q^*\alpha\right)\right) \in \Pic(\sQ^\delta(v)). \]

Consider $\gamma\in\Gamma'$, $\delta_\gamma\in \Gamma$, and the corresponding loci $W^\pm_{\gamma}$ of critical quotients. Since the determinant construction commutes with base change, we can compute $\lambda_\sE(\alpha)_{|W^\pm_{\gamma}}\simeq \lambda_{\sE_{|W^\pm_{\gamma}}}(\alpha)$. By Prop. \ref{prop_descr_of_walls_surfaces}, the fibers of $W^\pm_\gamma$ over $M(v)\times \Hilb(X,P)$ are projective spaces. 

Our goal is to compute the restriction of certain determinant line bundles to those fibers. We illustrate the computation of the restriction to $F^-=\PP(\Hom(B,\sO_Z)$. 

By the Grothendieck-Riemann-Roch theorem, $\det \left(p_!\left(\left[\sE\right]\cdot q^*\alpha\right)\right)$ coincides with the codimension one part of the push forward of $\ch{}(\sE)\cdot q^*(\alpha \cdot \td(X))$ to $F^-$.
Since $X$ has dimension 2, one needs to consider codimension 3 cycles pushing forward to codimension 1 cycles. Let $H$ be the hyperplane class in $F^-$, and denote by $h$ its pull-back, the cycles we are interested in are of the form $h\cdot \psi$, where $\psi$ is the pull back of a zero dimensional cycle on $X$. The universal sequence on $F^-$ reads:
\[ 0\to B\boxtimes \sO_{F^-}(-1) \to q^*\sO_Z \to \sQ  \to 0 \]
where $\sQ$ is the universal quotient restricted to $\PP^-$. Therefore we have $$\ch{}(\sE)= q^*(\ch{}(\sO_X) - \ch{}(\sO_Z)) + q^*\gamma\cdot \exp(-h).$$

Then, the relevant coefficient of $h$ in $\ch{}(\sE)\cdot q^*(\alpha\td(X))$ is exactly the codimension two part of $-\gamma\cdot \alpha\td(X)$, which by Hirzebruch-Riemann-Roch coincides with $-\chi(\gamma,\alpha)$.

An identical computation starting with the the universal sequence in $W^+_\gamma$, which reads
\[ 0\to \sE \to q^*I_Z \to B[1]\boxtimes \sO_{\PP^+}(1)  \to 0, \]
shows that $\lambda_\sE(\alpha)$ restricts to $\chi(\gamma,\alpha)H$ in $F^+$.

Like in Section \ref{sec_Bradlow_pairs_on_curves}, we choose two cohomology classes $\alpha_{min}$ and $\alpha_e$ chosen to be orthogonal, respectively, to $(1,D,c)$ and $e$. Then we define the line bundles 
\begin{equation*}
\sM_{1}:= \lambda_{\sE}(-\alpha_{min})  \qquad 
\sM_{2}:= \lambda_{\sE}(-\alpha_{e})
\end{equation*}

The computation above shows
\begin{lemma}\label{lem_GRR_computation_surfaces}
For $\delta_\gamma\in\Gamma$, the line bundles $\sM_1$ and $\sM_2$ restrict to $\chi(v,\alpha_{min})H$ and $\chi(v,\alpha_e)H$ on the fibers $F^-_\gamma$, and to the opposites on $F^+_\gamma$.
\end{lemma}

Notice that we do not claim that the critical loci $W_\gamma^\pm$ are projective bundles, like in the case of curves. This is why we carry out the computation above fiberwise, and not on the whole $W_\gamma^\pm$. An analog of Prop. \ref{prop_Bradlow_line_bundles} would require a more careful analysis of the critical loci, which we don't do in this note. 

\section{Further questions and applications}\label{sec_further_applications}

This section outlines two possible directions of investigation for generalized Quot schemes on surfaces. We provide some evidence motivating the importance of these examples.

\subsection{Analog of Bradlow pairs on surfaces}\label{sec_Th_pairs_on_surfaces}
In this section we work on a smooth projective surface $S$, and assume that $S$ contains no curve of negative self-intersection. We consider a class of quotients of the form
\begin{equation}
    \label{eq_quotient_O_to_L2}
    \sO_S \to L[2],
\end{equation}
where $L$ is a line bundle with sufficiently negative degree. Observe that these maps are a two-dimensional analog of Bradlow pairs on a curve $C$:
\[ E(-D) \to \sO_C \to \sO_C(-D)[1], \]
(see case (\ref{itm_map_Qd-1_to_determinants}) of Prop. \ref{prop_boundary_cases}).

We aim to interpret maps \eqref{eq_quotient_O_to_L2} as quotients in an abelian category obtained with a double tilt of $\Coh(S)$. Given an ample class $H$ on $S$ and parameters $\alpha>0$ and $\beta\in\R$ one can construct a stability condition $\sigma_{\alpha,\beta}=(\Coh^\beta(S),Z_{\alpha,\beta})$ as shown in \cite{AB13} (see also \cite{MS17}). Its slope function is given by 
\[  \nu_{\alpha,\beta}(E)= \frac{\ch 2^\beta(E)-\frac{\alpha^2}{2}\ch 0^\beta(E)H^2}{\ch 1^\beta(E)H}   \]
(where $\ch{}^\beta=e^{-\beta H}\ch{}$ is the twisted Chern character). 

\begin{definition}
Let 
\[
\sT^\theta_{\alpha,\beta}=\set{E\in \Coh^\beta(S) \st \mbox{ any quotient $E\twoheadrightarrow G$ satisfies $\nu_{\alpha,\beta}(G)>\theta$}}
\]
\[
\sF^\theta_{\alpha,\beta}=\set{E\in \Coh^\beta(S) \st \mbox{ any subobject $F\hookrightarrow E$ satisfies $\nu_{\alpha,\beta}(F)\leq\theta$}}
\]

and define $\sA^\theta_{\alpha,\beta}=\pair{\sF^\theta_{\alpha,\beta}[1],\sT^\theta_{\alpha,\beta}}$ to be the corresponding tilt of $\Coh^\beta(S)$. We use $\tau^\theta_{\alpha,\beta}$ to denote the corresponding $t$-structure. 
\end{definition}

Let $v=(1,\ch 1(L),\frac{L^2}{2})\in H^*(S)$ denote the Chern character of a line bundle $L$ on $S$ of degree $-d\leq 0$. For a map \eqref{eq_quotient_O_to_L2} to be in $\sA^\theta_{\alpha,\beta}$, a necessary condition is that $\sO_S\in \sT^\beta \cap \sT^\theta_{\alpha,\beta}$, and that the class $v$ satisfies $\mu(v)\leq \beta H^2$ and $\nu_{\alpha,\beta}(v)\leq \theta$. This can only be arranged if  $0>\beta H^2 >-d$ and 
\[ \nu_{\alpha,\beta}(-v)=\frac{d\beta + \frac{L^2}{2} +  H^2(\frac{\beta^2}{2}-\frac{\alpha^2}{2})}{-d-\beta H^2} < \frac{ H^2(\frac{\beta^2}{2}-\frac{\alpha^2}{2})}{-\beta H^2}=\nu_{\alpha,\beta}(\sO_S).  \]
This rewrites as 
\begin{align*}
   \frac{d\beta + \frac{L^2}{2} +  (-\beta H^2)\nu_{\alpha,\beta}(\sO_S)}{-d-\beta H^2}  < \nu_{\alpha,\beta}(\sO_S)\\
   d\left(\beta + \frac{L^2}{2d} + \nu_{\alpha,\beta}(\sO_S)\right)>0.
\end{align*}

We summarize the above in the following:
\begin{setup}\label{setup}
Let $S$ be a smooth projective surface containing no negative curves.
Fix a Chern character $v=\ch{}(L)$ where $L$ has degree $-d<0$, with $d$ large enough to have $H^0(S,L)=0$ (and therefore $\abs{K_SL^{-1}}\neq \emptyset$). Choose parameters $\beta$ and $\alpha$ such that:
\begin{enumerate}[(i)]
    \item $-d < \beta \ll 0$;
    \item $\alpha > 0$ such that $d\beta + \frac{L^2}{2d} + d\nu_{\alpha,\beta}(\sO_S))>0$;
\end{enumerate}
\end{setup}

For the rest of the section we assume \ref{setup} and consider the Quot schemes 
\[\sQ^\theta(v)\coloneqq \Quot_{S/\text{pt}}^{\tau^\theta_{\alpha,\beta}}(\sO,v)(\text{pt}).\] Denote $\theta_0\coloneqq \nu_{\alpha,\beta}(v)$. Note that the image of $\nu_{\alpha,\beta}$ is discrete, and the few next bigger values after $\theta_0$ are $\theta_k\coloneqq \nu_{\alpha,\beta}(v+(0,0,k))=\nu_{\alpha,\beta}(-v) + \frac{k}{d+\beta H^2}$ (as a consequence of the choices of $\alpha$ and $\beta$: every increment in $\ch0$ or $\ch 1$ would change $\nu_{\alpha,\beta}$ significantly). Now we describe $\sQ^\theta(v)$ and a wall-crossing as $\theta$ varies. 

First, we record a Lemma for future use:

\begin{lemma}[{\cite[Cor. 3.3]{Noo09}}]\label{lem_Noohi_surjectivity}
Suppose $X\in\sT^\theta_{\alpha,\beta}$ and $Y\in\sF^\theta_{\alpha,\beta}$, so that $Y[1]\in\sA_{\alpha,\beta}^\theta$. A map $\epsilon\colon X\to Y[1]$ corresponds to an extension $Y\to K \to X$ in $\Coh^\beta(S)$. Then $\epsilon$ is surjective in $\sA_{\alpha,\beta}^\theta$ if and only if $K \in \sT^\theta_{\alpha,\beta}$. 
\end{lemma}

\begin{proposition}
Under the assumptions \ref{setup}, pick $\theta_0 \leq\theta<\theta_1$. Then $\sQ^\theta(v) \simeq \abs{K_S L^{-1}}$.
\end{proposition}

\begin{proof}Consider a quotient $\set{\sO_S \to N}\in \sQ^\theta(v)$. Our choice of $\theta$ implies that $N=M[1]$ where $M\in\Coh^\beta(S)$ in a $\sigma_{\alpha,\beta}$-semistable object of class $-v$ (argue as in the proof of Prop. \ref{prop_boundary_cases}).
Since $S$ has no curve of negative self intersection, \cite[Thm. 1.1]{AM16} implies that $\sigma_{\alpha,\beta}$-semistable objects of class $-v$ are exactly shifts of line bundles of class $v$ (such objects are the only $\sigma_{\alpha,\beta}$-semistable objects in the Large Volume Limit, and \cite{AM16} shows that there are no walls for their class).
%Then by \cite[Lemma 6.18]{MS17}, $H^{-1}M$ is torsion free and $H^0(M)$ is supported on points. 
In other words, we must have $M\simeq L[1]$ for some line bundle $L$ of class $v$.
%If $H^{-1}M$ was only a torsion free sheaf, then there would be a destabilizing inclusion 
%\[ \C_p \to H^{-1}M[1] \]
%so $H^{-1}M$ is a line bundle. Moreover, if $H^0M\neq 0$, then $-v$ would not be of the form $-v=\exp(\cc{1}(\sL))$ where $\sL$ is any line bundle, which contradicts our choice. 
%Denote by $\pi\colon \sQ^\theta(v)\to \Pic^{-d}(S)$ the assignment $q\mapsto L$.

On the other hand, every non-zero map $q\colon \sO_S\to L[2]$ is surjective in $\sA^\theta_{\alpha,\beta}$: the cone $C$ fits in a non-split exact sequence in $\Coh^\beta(S)$
\[ L[1]\to C \to \sO_S \]
so $\nu_{\alpha,\beta}(C)>\theta$. If $C\notin \sT^\theta_{\alpha,\beta}$ then it would split the sequence by our choice of $\theta$, so we can conclude by Lemma \ref{lem_Noohi_surjectivity}). 
As a consequence, the Quot scheme coincides with the whole projective space  $ \PP(\Hom(\sO,L[2]))\simeq \abs{K_S L^{-1}}$.
\end{proof}

Much like in the curve case, we aim to provide a geometric interpretation of wall-crossing phenomena for $\sQ^\theta(v)$. We recall Terracini's lemma and the geometric construction which will be needed later. Keep the notation as above. Then the linear series $\abs{K_SL^{-1}}\simeq \pr N$ contains a copy of $S$ embedded with the map sending $p\in S$ to the composition $\sO_S \to \C_p \xrightarrow{a} L[2]$ (here the first map is restriction of $\sO_S$, and $a\in \Hom(\C_p,L[2])\simeq \Hom(L,\C_p\otimes \omega_S)^*$ is the dual of the restriction of $L$ to its fiber over $p$). In other words, $p\in S\subset \abs{K_SL^{-1}}$ coincides with the one dimensional quotient $H^0(K_SL^{-1})/ H^0(K_SL^{-1}I_p)$.

Thus, we can identify the tangent space $T_p\pr N$ with $H^0(K_SL^{-1}I_p)$. Similarly, if $p\neq q$ are two points of $S$, the secant line through $p$ and $q$ in $\pr N$ is identified by the two dimensional quotient $H^0(K_SL^{-1})/ H^0(K_SL^{-1}I_p\otimes I_q)$. Finally, we recall:

\begin{lemma}[Terracini]
The normal space $N_pS$ is identified with $H^0(K_SL^{-1}I_p^2)$.  
\end{lemma}

\begin{comment}
In fact, $h^0(K_SL^{-1})-3=N-2=\dim \pr N - \dim S$: one can compute this tensoring the sequence $I_p \to \sO_S \to \C_p$ by $AI_p$, where $A$ is a sufficiently positive line bundle. The result is
\[
\C_p \to AI_p^{\otimes 2}\to AI_p \to A I_p/I_P^2
\]
(here $\C_p=\sT or_1(I_p,\C_p)$) which breaks in two short exact sequences, the first being
\[
\C_p \to AI_p^{\otimes 2} \to AI_p^{2}
\]
Take global sections and see that thedimension of $I_p^2$ is one smaller, so the codimension is 3 (one higher than 2).
\end{comment}

%Now, fix $L\in\Pic^{-d}(S)$ and denote $\sQ^\theta_L(v)\coloneqq \pi^{-1}(L)$. 
The first wall crossing corresponds to the value $\theta_1=  \nu_{\alpha,\beta}(-v) + \frac{1}{d+\beta H^2}$. 

\begin{proposition}\label{prop_wall_terracini}
For $\theta_1\leq \theta < \theta_2$, then $\sQ^\theta(v)$ has one irreducible component $\sQ'$ which coincides with the blowup of $\abs{K_S L^{-1}}$ along $S$, embedded as above. This component intersects the others in the exceptional locus of the blowup.
\end{proposition}

\begin{proof}

The first wall is given by quotients $\sO\to L[2]$ which factor through a map 
\[ (\ast)\colon \ \ \sO\to M \to L[2] \] where $[M]=[\C_p]$. We claim that in fact, $M$ has to be a skyscraper sheaf. In other words, the quotients $(\ast)$ correspond exactly to the points of $S$ in $\abs{K_SL^{-1}}$. 

Indeed, the map $a\colon M\to L[2]$ is surjective, hence the shift of its cone $K=\ker a$ belongs to $\sA^\theta_{\alpha,\beta}$. By our choice of $\theta$, this implies that $K$ is $\sigma_{\alpha,\beta}$-semistable. By \cite[Lemma 6.18]{MS17}, we must have that $H^{-1}K$ is torsion free and $\mu$-semistable and $H^0(K)$ is supported on points. Moreover, $H^{-2}(K)=0$ since $K\in \Coh^\beta(S)$.
Write $\ch{}(H^0(K))=(0,0,s)$ with $s\geq 0$. Then since 
\[ \ch{}(H^0(K))-\ch{}(H^{-1}(K))=\ch{}(K)=-\ch{}(L)+(0,0,1) \]
we must have $\ch{}(H^{-1}(K))=(1,\ch{1}(L),t)$ where $-t+s=-\ch{2}(L) + 1$.
Since $H^{-1}K$ is torsion free and $\mu$-semistable, it satisfies the Bogomolov inequality 
\[  t\leq \frac{\ch{1}(H^{-1}(K))^2}{2\rk(H^{-1}(K))} = \ch{2}(L),  \]
whence $s\leq 1$. From the long exact sequence of cohomologies we get
\[ 0\to H^{-2}(M) \to L \xrightarrow{\alpha} H^{-1}(K) \to H^{-1}(M) \to 0\]
and an isomorphism $H^0(K)\simeq H^0(M)$. If $s=0$, then $\ch{}(H^{-1}(K))=\ch{}(L)-1$ and the map $\alpha$ must vanish (otherwise, it would be injective with cokernel of negative length). This implies $K\simeq M$ and $a=0$, which is a contradiction. If $s=1$, then $H^0(M)\simeq \C_p$ for some $p$. Moreover, $\ch{}(H^{-1}(K))=\ch{}(L)$, which implies that $\alpha$ is an isomorphism, $M\simeq H^0(M)$, and $K=\mathsf D(L^{-1}I_p)[-1]$. This proves our initial claim.

The quotients appearing after the wall crossing are extensions in $\Ext^1(\C_p, \mathsf D(L^{-1}I_q))=0$. In other words we have
\[\sO\twoheadrightarrow \C_p \oplus \mathsf D(L^{-1}I_q)\eqqcolon Q\]
and the snake lemma
\begin{center}$(**)\quad$
\begin{tikzcd}[row sep=small]
& &\mathsf D(L^{-1}I_q) \dar\\
E \dar \rar & \sO \ar[equal]{d} \rar & Q \dar \\
I_p \rar \dar  &\sO \rar & \C_p \\
\mathsf D(L^{-1}I_q) & &
\end{tikzcd}
\end{center}
shows that the datum of such a quotient is the datum of a 
surjective map $I_p\to \mathsf D(L^{-1}I_q)$. We claim that every map $b\colon I_p\to \mathsf D(L^{-1}I_q)$ is a surjection in $\sA^\theta_{\alpha,\beta}$: indeed, by our choice of $\theta$, the distinguished triangle 
\begin{equation}
    \label{eq_dist_triang_cone}
    \mathsf D(L^{-1}I_q)[-1] \to C(b) \to I_p 
\end{equation}
implies that $C(b)\in \sA^\theta_{\alpha,\beta}$, as long as the triangle $\eqref{eq_dist_triang_cone}$ doesn't split on any Jordan-H\"older factor of $\mathsf D(L^{-1}I_q)[-1]$. But $\mathsf D(L^{-1}I_q)[-1]$ is stable: by our choices of the parameters, $\im Z_{\alpha,\beta}$ attains its positive minimal value at $\mathsf D(L^{-1}I_q)[-1]$.

If $p=q$, the space 
\[\Hom(I_p,\mathsf D(L^{-1}I_q))\simeq H^0(I_p\otimes^{\mathbb L} K_SL^{-1}I_q)^*\]
admits a codimension 1 subspace $H^0(K_SL^{-1}I_p^2)^*$ (see Remark \ref{rmk_terracini_functors}), which is identified with $N_pS^*$ by Terracini's lemma. This implies that the moduli space after the first wall has an irreducible component isomorphic to $\Bl_S(\pr N)$, which intersects other components in the exceptional locus.
\end{proof}

\begin{remark}\label{rmk_terracini_functors}
The functor \eqref{eq_dual} and Serre duality yield an isomorphism of functors
\[\Hom(-,\mathsf D(L^{-1}I_q))\simeq H^0(-\otimes^{\mathbb L} K_SL^{-1}I_p)^*.\]
Tensoring $I_p\to \sO\to \C_p$ by $K_SL^{-1}I_p$ yields
\[ K_SL^{-1}\otimes \left[ Tor_1(I_p,\C_p) \to I_p\otimes I_p \xrightarrow{a} I_p \to I_p/I_p^2 \right]  \]
where $a$ factors thorugh $I_p^2$. Taking global sections on the short exact sequence $I_p^2 \to I_p \to I_p/I_p^2$ and dualizing, we get
\[ 0= H^1(K_SL^{-1}I_p^2)^*\to H^0(I_p/I_p^2)^* \to H^0(K_SL^{-1}I_p)^* \to H^0(K_SL^{-1}I_p^2)^* \to 0 \]
which is naturally identified with the sequence
\[
0\to T_pS \to T_p\pr N \to N_p S \to 0.
\]
\end{remark}

\begin{remark}
With similar arguments, one can focus on the component $\sQ'$ and study the next wall. Reasoning in analogy with the case of curves, we expect it to involve $\Sigma'$, the strict transform of the variety $\Sigma$ of secant lines of $S$ in $\sQ'$ (see \cite{Be97}). The goal would be to describe a diagram of birational transformations analog to \eqref{eq_thaddeus_diagram} in the context of algebraic surfaces.

The starting point is the following observation: the second wall involves quotients $\sO\to L[2]$ which factor through $\sO_Z$, where $Z$ is a length 2 subscheme of $S$. A quotient of this form  corresponds to a point on the secant line to $S$ through the points $\set{p_1,p_2}=\Supp Z$ (if $Z$ is reduced), or on the tangent line to $S$ determined by $Z$ (if $Z$ is non-reduced), but off the surface $S$. In other words, $\Sigma' \setminus E$ contains quotients that are critical for the second wall. The analysis of points in $\Sigma' \cap E$ gets more involved, and we leave it as an open question for future investigation.
\end{remark}

\subsection{Comparison with Pandharipande--Thomas stable pairs}\label{sec_PT_invariants}

In \cite{PT_curve_counting} (see also \cite[Appendix B]{PT_stable_pairs_BPS}), the authors define the moduli space $P(S)$ of stable pairs $(F,s)$ where $F$ is a sheaf of fixed Hilbert polynomial supported in dimension one and $s \in H^0(S,F)$ is a section. The stability condition is:
\begin{enumerate}
    \item the sheaf $F$ is pure;
    \item the section $s$ has 0-dimensional cokernel.
\end{enumerate}

PT pairs can be constructed as Quot schemes in a single tilt as in \cite[Lemma 2.3]{Bri11_Hall_curve_counting}, here we compare them with another Quot scheme of quotients in a double tilt.

Let $C$ be a curve of class $\gamma\in H_2(S,\Z)$ and denote by $\iota\colon C\to S$ the inclusion. A special class of stable pairs is given by $(\iota_*(K_CL_C^{-1}), s)$ where $L_C$ is a line bundle on $C$, and $s$ is induced by an element of $H^0(C, K_CL_C^{-1})$. Since $\iota$ is affine, we have 
\begin{equation}
\label{eq_chain_of_iso_PT}
    \begin{split}
  H^0(S,\iota_*(K_CL_C^{-1}))\simeq H^0(C,K_CL_C^{-1}) \simeq H^1(C,L_C)^*\\
   \simeq H^1(S,\iota_*L_C)^*\simeq \Hom(\sO_S,\iota_*L_C[1])^*.
    \end{split}
\end{equation}

These isomorphisms suggest a relation between PT pairs and maps 
\begin{equation}
    \label{eq_PT_stable_pair}
q\colon \sO_S\to \iota_*L_C[1].
\end{equation}

We now find necessary conditions to interpret $q$ as a quotient in $\sA_{\alpha,\beta}^\theta$. Let $d$ be the degree of $L_C$, we write $e\coloneqq d-\frac{\gamma^2}{2}$ and assume $d\ll 0$. Then set
$$v\coloneqq\ch{}(\iota_* L_C)=(0,\gamma,e).$$
As before, a necessary condition for a map $q$ to be in $\sA_{\alpha,\beta}^\theta$ is that $\sO_S\in\sT^\beta \cap\sT_{\alpha,\beta}^\theta$, and $\nu_{\alpha,\beta}(v)\leq \theta$ (since $\mu(v)=+\infty$, so that a sheaf of class $v$ cannot belong to $\sF^\beta$ for any $\beta$). In other words, the necessary conditions are $\beta<0$ and $\nu_{\alpha,\beta}(v) < \nu_{\alpha,\beta}(\sO_S)$. The inequality reads
\[ \frac{e}{\gamma H}-\beta < \frac{\beta^2 - \alpha^2}{-2\beta}  \]
Rewrite this as
\[ 2\beta^2 - 2\frac{\beta e}{\gamma H} < \beta^2 - \alpha^2\]
\[\alpha^2 + \left(\beta -\frac{e}{\gamma H}\right)^2 < \left(\frac{e}{\gamma H}\right)^2\]
the equation of a circle centered at the point $(\alpha=0,\beta=\frac{e}{\gamma H})$. Then, fix parameters $\alpha_0,\beta_0= \frac{e}{\gamma H}$ near the center of this circle: they satisfy the necessary conditions.

As an example, we describe the Quot scheme one obtains under the simplifying assumption that $\gamma$ is an indecomposable effective curve class class satisfying $\gamma.H=1$ and $\gamma^2=1$. In this case, it is easy to classify $\sigma_{\alpha,\beta}$-semistable objects of class $v$ for $\alpha>0$ and $\beta_0=e$. In fact, we have 

\begin{lemma}\label{lem_wall-crossing}
For all $\alpha>0$, the $\sigma_{\alpha,e}$-semistable objects of class $v$ are pure one-dimensional sheaves supported on a curve of class $\gamma$. Moreover, all these objects are stable.
\end{lemma}
\begin{proof}
The usual computation for walls (see for example \cite[Prop. 2.9]{APR19}) shows that a potential destabilizer $F'$ must have $\ch 1^\beta(F')H=0$ by our assumption on $\gamma$. This implies that $F'\in \PP(1)$ is always destabilizing for all values of $\gamma$. In other words, if $F$ is $\sigma_{\alpha,e}$-semistable for some $\alpha>0$, then it is semistable for all $\alpha > 0$. By the Large Volume Limit, $F$ must then be a pure one dimensional sheaf supported on a curve of class $\gamma$. Stability follows because $\im Z_{\alpha,e}(F)$ is the minimum for $\im Z_{\alpha,e}$.
\end{proof}

Now we give a description of $\sQ^\theta(v)\coloneqq \Quot_{S/\text{pt}}^{\tau^\theta_{\alpha_0,e}}(\sO,v)(\text{pt})$ in the case when $\theta=\nu_{\alpha_0,e}(v)$, under the additional assumptions on $\gamma$. Let $\sM$ denote the moduli space of pure dimension 1 subschemes of $S$ of class $\gamma$ in $S$. Then $\sM$ is a Hilbert scheme of curves, denote by $\sC$ and $\Pic^d(\sC)$ the universal curve and the universal Jacobian, respectively. Consider the Cartesian square

\begin{center}
\begin{tikzcd}[row sep=small]
\sC \times_\sM \Pic^d(\sC) \rar{\pi} \dar{\rho} & \Pic^d(\sC) \dar\\
\sC \rar & \sM 
\end{tikzcd}
\end{center}
\begin{comment}
and the vector bundle \AAA{check that this is the correct bundle, and not the dual...}
%it is: check cohomology vanishing on each fiber given by degree constraints
$\mathcal E \coloneqq \dR^1\pi_*(\PP)$ on $\Pic^d(\sC)$, where $\PP$ denotes the Poincar\'e line bundle on $\sC \times_\sM \Pic^d(\sC) $. The fiber of the projective bundle  $\mathbb P (\sE)$ above a point $(C,L_C)\in \Pic(\sC)$ is $\mathbb P(H^1(C,L_C))\simeq \mathbb P(\Hom(\sO_S,\iota_*L_C[1]))$. 
\end{comment}
and the vector bundle $\mathcal E \coloneqq \pi_*(\rho^*\omega_{\sC/\sM} \otimes \PP^{-1})$ on $\Pic(\sC)$, where $\omega_{\sC/\sM}$ is the relative canonical bundle of $\sC\to \sM$, and $\PP$ denotes the Poincar\'e line bundle on $\sC \times_\sM \Pic(\sC) $. The fiber of the projective bundle  $\mathbb P (\sE)$ above a point $(C,L_C)\in \Pic(\sC)$ is $\mathbb P(H^0(C,K_C L_C^{-1}))\simeq \mathbb P (\Hom(\sO_S,\iota_*L_C[1])^*)$.

%This implies that there is a morphism
%\[\sC \times_\sM \Pic(\sC) \to \mathbb P(\sE) \]
%over $\Pic(\sC)$, which restricts to $\phi_{K_CL^{-1}}$ above every point $(C,L_C)\in \Pic(\sC)$.

%This implies that there is a morphism
%\[\sC \times_\sM \Pic(\sC) \to \mathbb P(\sE) \]
%over $\Pic(\sC)$, which restricts to $\phi_{K_SL^{-1}}$ above every point $(C,L_C)\in \Pic(\sC)$.

\begin{proposition}\label{prop_PT_moduli_space}
Let $\theta=\nu_{\alpha_0,e}(v)$. Then the Quot scheme $\sQ^\theta(v)$ is isomorphic to $\mathbb{P}(\sE)$.
%scheme $\sQ^\theta(v)\coloneqq \Quot_{\sA^\theta_{\alpha_0,e}}(\sO,v)$ is isomorphic to the space $\mathbb P(\sE)$. 
\end{proposition}

\begin{proof}
For every closed point $t\in \sQ^\theta(v)$ we get a map $q_t\colon \sO_S \to \sL_t[1]$. The chosen value of $\theta$ implies that $\sL_t[1]$ is $\sigma_{\alpha_0,e}$-semistable. Then, by Lemma \ref{lem_wall-crossing}, $\sL_t$ is a pure one-dimensional sheaf supported on a curve $C$ of class $\gamma$, so every $q_t$ has the form \eqref{eq_PT_stable_pair}. Then, $q_t$ identifies a one-dimensional quotient $\Hom(\sO_S,\iota_*L_C[1])^*\twoheadrightarrow [q_t]$, and the assignment
\[q_t \mapsto [q_t]\]
defines a map $\sQ^\theta(v)\to \mathbb P(\sE)$. Conversely, $\mathbb P(\sE)$ parametrizes quotients of $\sO_S$ of class $v$: observe that all non-trivial maps $q\colon \sO_S\to \iota_*L_C[1]$ are surjections in $\sA^\theta_{\alpha_0,e}$, since stability of $\iota_*L_C[1]$ (Lemma \ref{lem_wall-crossing}) implies that the cone $C(q)=K[1]$ satisfies $K\in \sT^\theta_{\alpha,\beta}$ (cfr. the proof of Prop \ref{prop_wall_terracini}). Hence, $\mathbb P(\sE)$ admits a morphism to $\sQ^\theta(v)$. It is straightforward to check that these morphisms are inverses of one another. 
%To see that the map is an isomorphism, observe that all non-trivial maps $q\colon \sO_S\to \iota_*L_C[1]$ are surjections in $\sA^\theta_{\alpha_0,e}$: stability of $\iota_*L_C[1]$ (Lemma \ref{lem_wall-crossing}) implies that the cone $C(q)=K[1]$ satisfies $K\in \sT^\theta_{\alpha,\beta}$ (cfr. the proof of Prop \ref{prop_wall_terracini}). 
\end{proof}

%\printbibliography
\bibliographystyle{amsalpha}
\bibliography{./bibliography1}

\end{document}